\documentclass[reqno,3p]{elsarticle} %

\usepackage{hyperref}

\hypersetup{pdfauthor={Matthias Lenz},pdftitle={Hierarchical Zonotopal Power Ideals}, pdfsubject={Zonotopal Algebra}}
 
\usepackage[T1]{fontenc}
\usepackage[ansinew]{inputenc}
\usepackage{amssymb} 
\usepackage{amsthm}   
\usepackage{amsmath} 
\usepackage{amsfonts} 

\usepackage{upref} %
\usepackage{graphicx}
\usepackage{color}
\usepackage{url}
\usepackage{tabularx} %
\usepackage{paralist} %

\usepackage{calc}

\usepackage{bm} %

\newtheorem{Theorem}{Theorem}[section]
\newtheorem{Corollary}[Theorem]{Corollary}
\newtheorem{Proposition}[Theorem]{Proposition}
\newtheorem{Lemma}[Theorem]{Lemma}

\theoremstyle{definition}
\newtheorem{Definition}[Theorem]{Definition}
\newtheorem{Example}[Theorem]{Example}

\theoremstyle{remark}
\newtheorem{Remark}[Theorem]{Remark}

\newcommand{\vek}[1]{{\bm{#1}}} %
\newcommand{\abs}[1]{\left|#1\right|}
\newcommand{\dunion}{\dot{\cup}} %

\newcommand{\codim}{\mathop{\mathrm{codim}}} %
\newcommand{\rank}{\mathop{\mathrm{rk}}}

\newcommand{\BigDisjUnion}{\mathop{
\dot{\bigcup}
}\limits} %

\newcommand{\DisjUnion}{\mathop{
\dunion
}\limits} %

\newcommand{\gleich}[1]{\stackrel{\parbox{1cm}{\centering\scriptsize{#1}}}{=}}

\newcommand{\st}{s.\,t.\ } 
\newcommand{\ie}{\textit{i.\,e.\ }} 
\newcommand{\eg}{\textit{e.\,g.\ }} 

\newcommand{\N}{\mathbb{N}}
\newcommand{\Z}{\mathbb{Z}}

\newcommand{\R}{\mathbb{R}}
\newcommand{\C}{\mathbb{C}}
\newcommand{\K}{\mathbb{K}}

\newcommand{\Bcal}{\mathcal{B}}

\newcommand{\Dcal}{\mathcal{D}}
\newcommand{\Hcal}{\mathcal{H}}
\newcommand{\Ical}{\mathcal{I}}

\newcommand{\Zcal}{\mathcal{Z}}
\newcommand{\Pcal}{\mathcal{P}}

\newcommand{\eps}{\varepsilon}

\newcommand{\ideal}{\mathop{\mathrm{ideal}}}
\newcommand{\spa}{\mathop{\mathrm{span}}}
\newcommand{\BB}{ \mathbb B}
\newcommand{\Fcal}{\mathcal F}
\newcommand{\Lcal}{\mathcal L}
\newcommand{\Mfrak}{\mathfrak M}
\newcommand{\card}[1]{{\abs{#1}}}

\newcommand{\diff}[1]{\frac{\partial}{\partial #1}}

\newcommand{\hilb}{\mathop{\mathrm{Hilb}}}
\newcommand{\clos}{\mathop{\mathrm{cl}}\nolimits} 
\newcommand{\sym}{\mathop{\mathrm{Sym}}} 
\newcommand{\pair}[2]{\langle #1,#2 \rangle}

\numberwithin{equation}{section}

\newcommand{\sameterminologyasinmaindefinition}{%
We use the same terminology 
as in Definition~\ref{Definition:PIdefinition}. }

\journal{European Journal of Combinatorics}

\begin{document}

\begin{frontmatter}



\title{Hierarchical zonotopal power ideals}
\author{Matthias Lenz\fnref{label1}}
\ead{lenz@math.tu-berlin.de}
\ead[url]{http://page.math.tu-berlin.de/~lenz/}
\address{%
Technische Universit\"at Berlin \\
Sekretariat MA 3-6 \\
Stra\ss e des 17.~Juni 136 \\
10623 Berlin \\
GERMANY
}

\fntext[label1]{This work was supported by
 a Sofia Kovalevskaya Research Prize of Alexander von Humboldt Foundation awarded to Olga Holtz.}


\begin{abstract}
Zonotopal algebra deals with ideals and vector spaces of polynomials that are related %
 to  several combinatorial and geometric structures defined by a finite sequence of vectors. 
Given such a sequence $X$, %
an integer $k\ge -1$ and an upper set in the lattice of flats of the matroid defined by $X$, we define and study
the associated \emph{hierarchical zonotopal power ideal}.
This ideal is generated by powers of linear forms. Its Hilbert series depends only on the matroid structure of $X$.
Via the Tutte polynomial, it is related to various other matroid invariants,
\eg the shelling polynomial and the characteristic polynomial. 

This work unifies and generalizes results by Ardila-Postnikov on power ideals and by Holtz-Ron and Holtz-Ron-Xu on (hierarchical) zonotopal algebra.  We also generalize a result on zonotopal Cox modules that were introduced by Sturmfels-Xu.
\end{abstract}


\end{frontmatter}


\section{Introduction}
Let $X=(x_1,\ldots, x_N) \subseteq \R^{r}$ be a sequence of vectors that span $\R^r$. 
For a vector $\eta $, %
let  $m(\eta)$ denote the number of vectors in $X$ that are \emph{not}
perpendicular to  $\eta$.
A vector $v\in \R^r$ defines a linear polynomial $p_v:=\sum_i v_i t_i \in \R[t_1,\ldots, t_r]$.
For $Y\subseteq X$, let $p_Y:= \prod_{x\in Y} p_x$. Then define 
 
\begin{align}
 \Pcal(X) &:=\spa\{ p_Y : X\setminus Y \text{ spans } \R^r  \}  %
 \\
 \Ical(X) &:= \ideal %
  \{ p_\eta^{m(\eta)}  : \eta \neq 0\}
\end{align}
The following theorem and several generalizations are well-known:
\begin{Theorem}[\cite{ardila-postnikov-2009,boor-dyn-ron-1991,holtz-ron-2011}]
\label{Theorem:CentralMainTheorem}
 \begin{align}
   \Pcal(X) = \ker \Ical(X) := \spa \left\{ q \in \R[t_1,\ldots, t_r] : 
   f (D)q = 0 \text{ for all } f\in \Ical(X) \right\}     
 \end{align}
where $f(D):=f\left(\frac{\partial}{\partial t_1},\ldots, \frac{\partial}{\partial t_r}\right)$.
In addition,  $\Ical(X)$ is equal to the ideal
\begin{align*}
 \Ical'(X):=\{ p_\eta^{m(\eta)} : \text{the vectors in $X$ that are perpendicular to $\eta$ span a hyperplane}%
   \}
\end{align*}
\end{Theorem}
In this paper, we show that a statement as in  Theorem~\ref{Theorem:CentralMainTheorem} holds in a far more general setting:
we study the kernel of the \emph{hierarchical zonotopal power ideal} 
\begin{align}
\Ical(X,k,J):=\ideal\{ p_{\eta}^{m(\eta)+k + \chi_J(\eta)} : \eta  \neq 0 \}
\end{align}
 where 
$k\ge -1$ is an integer and $\chi_J$ is the indicator function of an upper set $J$ in the lattice of flats of the matroid defined by $X$. 
We examine those spaces in a slightly more abstract setting, \eg
 $\Pcal(X,k,J)$ is contained in the symmetric algebra over an $r$-dimensional $\K$-vector space, where  $\K$ is a field of characteristic zero.

\smallskip
The choice of a sequence of vectors $X$ defines a large number of objects in 
various mathematical fields which are all related to zonotopal algebra \cite{holtz-ron-2011}. Examples include
combinatorics (matroids, matroid and graph polynomials, generalized parking functions and chip firing games  if $X$ is graphic \cite{desjardins-2010,holtz-ron-parking,merino-2001,postnikov-shapiro-2004}), discrete geometry (hyperplane arrangements, zonotopes, and tilings of zonotopes), approximation theory (box splines \cite{BoxSplineBook}, least map interpolation) 
 and algebraic geometry (Cox rings, fat point ideals 
 \cite{ardila-postnikov-2009,geramita-schenck-1998,sturmfels-xu-2010}). 
Central $\Pcal$-spaces  (in our terminology the kernel of $\Ical'(X)$) were introduced in the literature on approximation theory around 1990 \cite{akopyan-saakyan-1988, boor-dyn-ron-1991, dyn-ron-1990}. A dual space called $\Dcal(X)$
appeared almost 30 years ago \cite{boor-hoellig-1982}. 
See \cite[Section 1.2]{holtz-ron-2011} for a historic survey and the book \cite{BoxSplineBook} for a treatment of polynomial
spaces appearing in the theory of box splines.

Recently, Olga Holtz and Amos Ron coined the term \emph{zonotopal algebra}  \cite{holtz-ron-2011}. 
They introduced internal ($k=-1$) and external ($k= +1$) $\Pcal$-spaces and $\Dcal$-spaces.
Federico Ardila and Alexander Postnikov \cite{ardila-postnikov-2009} constructed  $\Pcal$-spaces for arbitrary integers $k\ge -1$.
Olga Holtz, Amos Ron, and Zhiqiang Xu \cite{holtz-ron-xu-2012} introduced hierarchical zonotopal spaces, 
\ie structures that depend 
on the choice of an upper set $J$ in addition to $X$ and $k$. They 
 studied semi-internal and semi-external spaces (\ie $k=-1$ and $k=0$ and some special upper sets $J$). 
A result related to semi-external spaces 
 (a decomposition of the external space in terms of the lattice of flats of $X$) 
 appeared in the work of Peter Orlik and Hiroaki Terao \cite{orlik-terao-1994} on hyperplane arrangements.
The central case was treated in an algebraic setting in \cite{procesi-concini-2006,concini-procesi-book}.
Other related results include \cite{berget-2010,lenz-forwardexchange,li-ron-2011,wagner-1999}. %

In algebra, kernels of ideals where studied by Macaulay already in 1916 under the name \emph{inverse system}
\cite{groebner-1938,abramson-2010,macaulay-1916}.
The first use of inverse systems to study powers of linear forms appeared in a 1995
paper by Jacques Emsalem   and Anthony Iarrobino \cite{emsalem-iarrobino-1995}. 
 A second paper exploring
the connection between powers of linear forms, fatpoints
and splines by
Anthony Geramita and Henry Schenck 
  appeared in 1998 \cite{geramita-schenck-1998}.

The \emph{least map} \cite{deBoor-Ron-1992} assigns to a finite set $S\subseteq\R^r$ of cardinality $m$ an $m$-dimensional space of homogeneous polynomials in $\R[t_1,\ldots, t_r]$. 
Holtz and Ron \cite{holtz-ron-2011} showed that in the internal, central and external case, $\Pcal$-spaces can be obtained via the least map if $X \subseteq \Z^r$ is unimodular\footnote{%
This means that if we interpret the elements of $X$ as the columns of a matrix,
all square submatrices have determinant $0$, $1$, or $-1$.}.
In those cases, the $\Pcal$-spaces are obtained by choosing the set $S$ as
 a certain subset of the set of integral points in the zonotope 
 \begin{align}
 Z(X):= \left\{ \sum_{i=1}^N \lambda_i x_i : 0\le \lambda_i \le 1  \right\}
 \end{align}
In a follow-up paper, we will investigate a similar correspondence in the hierarchical setting.
There is also a discrete theory, where differential operators are replaced by difference operators
(\cite{concini-procesi-book, moci-tutte-2012}, see \cite{moci-adderio-2011} for a related combinatorial structure). In this theory, the connection with lattice points in the zonotope remains valid even if $X$ is not unimodular.
\smallskip

As an example for the connections between zonotopal algebra and combinatorics,
we now explain various relationships between zonotopal spaces and 
 matroid/graph polynomials. 
They can be deduced from the fact that both, the matroid/graph polynomials \cite{ellis-merino-2011} 
and the Hilbert series of the zonotopal spaces \cite{ardila-postnikov-2009} are evaluations of the Tutte polynomial (see also equations \eqref{eq:CentralTutteEvaluation} and \eqref{eq:ExternalTutteEvaluation}).

$\hilb(\Pcal(X,0,\{X\}), t)$ equals the shelling polynomial \cite{bjoerner-1992} $h_{\Delta(X^*)}(t)$  
of the matroid dual to $X$ with the coefficients reversed. The $t^{N-r-i}$ coefficient of $h_{\Delta(X^*)}(t+1)$ 
equals the number of independent sets of cardinality $i$ in the matroid $X$.
Let $X_G$ denote a reduced oriented incidence matrix of a connected graph $G$, \ie a matrix that is obtained from an oriented incidence matrix by 
 deleting one row so that it has full rank. 
The Hilbert series of $\Pcal(X_G,-1,\{X_G\})$
 is related to the flow polynomial  $\phi_G$ (see \cite[Section 6.3]{ellis-merino-2011} for a reference). 
 By duality it is also related to the chromatic  polynomial of the graph, resp.\ the characteristic polynomial of the matroid in the general case.
The connection is as follows: if $G$ is connected, then 
\begin{align}
\phi_G(t)=(t-1)^{N-r}\hilb(\Pcal(X_G,-1,\{X_G\}),1 /(1-t))
\end{align}
The four color theorem is equivalent to the following statement:
if $G$ is a planar graph, then  
 \begin{align}
  \hilb(\Pcal(X_{G^*},-1,\{X_{G^*}\}), -1/3) > 0,
 \end{align}
  where $G^*$ denotes the graph dual to $G$.
Further connections between zonotopal algebra and 
 matroid/graph polynomials are explained in detail in \cite{lenz-2011}.

\paragraph{Organization of the article}

In Section~\ref{section:Background}, we introduce our notation and review the mathematical background. 
In Section~\ref{section:MainTheorem}, we describe the kernels of the ideals $\Ical(X,k,J)$
and define a subideal $\Ical'(X,k,J)$ with finitely many generators. We show that for $k \le 0$, 
the two ideals are equal.

In Section~\ref{section:basis}, we construct bases for the vector spaces $\Pcal(X,k, J)$ for $k\ge 0$. 
We deduce formulas for the Hilbert series of the spaces $\Pcal(X,k, J)$ in Section
\ref{section:HilbertSeries}. 
Those formulas depend only on the matroid structure of $X$ but not on the representation.
In Section~\ref{section:cox}, we apply our results to prove a statement about zonotopal Cox
modules that were defined  by Bernd Sturmfels and Zhiqiang Xu \cite{sturmfels-xu-2010}.
Finally, in Section~\ref{section:examples}, we give plenty of examples.

\subsection*{Acknowledgments}
I would like to thank my advisor Olga Holtz for fruitful discussions of this work. %
I also thank
Federico Ardila who suggested the possibility of the results in 
Section~\ref{section:cox} and Andrew Berget who made some valuable remarks.
An extended abstract of this paper has appeared in the proceedings of FPSAC 2011
\cite{lenz-fpsac-2011}.

\section{Preliminaries}
\label{section:Background}

\paragraph{Notation}
\label{subsect:Notation}

The following notation is used throughout this paper:
$\N:=\{0,1,2,\linebreak[2] 3,\ldots\}$. For $n\in \N$, let $[n]:=\{1,\ldots, n \}$.
 $\K$ is a fixed field of characteristic zero. $V$ denotes a finite-dimensional $\K$-vector space of dimension $r\ge 1$ and $U:=V^*$ its dual.
Our main object of study is a finite sequence $X=(x_1,\ldots, x_N) \subseteq U$ whose elements span $U$.
 We slightly abuse notation by using the symbol $\subseteq$ for subsequences. 
For $Y\subseteq X$,  $X\setminus Y$ 
 denotes the deletion of a subsequence, 
\ie $(x_1,x_2)\setminus (x_1)=(x_2)$ even if $x_1=x_2$.
The order of the elements in $X$ is irrelevant for us except in a few cases, where this is explicitly mentioned.

\paragraph{Matroids and posets}

Let $X=(x_1,\ldots, x_N)$ be a finite sequence whose elements span $U$.
Let $\Mfrak(X):=\{ I \subseteq [N] %
: (x_i : i \in I)  \text{ linearly independent} \}$. Then 
$\Mfrak(X)$ is a %
 matroid of rank $r$ on  $N$ elements. %
$X$ is called a \emph{$\K$-representation} of the matroid $\Mfrak(X)$.
For more information an matroids, see Oxley's book \cite{MatroidTheory-Oxley}.

We now introduce some additional matroid-theoretic concepts. To facilitate notation, we mostly write $X$ instead of $\Mfrak(X)$.

The \emph{rank} of $Y\subseteq X$ is defined as the cardinality of a maximal linearly independent set 
contained in $Y$. It is denoted $\rank(Y)$. 
The \emph{closure} of $Y$ in $X$ is defined as $\clos_X(Y):=\{ x\in X : \rank(Y \cup x)=\rank(Y) \}$.
$C\subseteq X$ is called a \emph{flat} if $C=\clos(C)$. 
A \emph{hyperplane} is a flat of rank $r-1$.
The set of all hyperplanes in $X$ is denoted by $\Hcal = \Hcal(X)$.

Let $C\subseteq X$ be a flat and $\eta\in V$. If $C= \{x\in X : \eta(x)=0\}$, we call $\eta$ a \emph{defining normal} for $C$. 
Note that for hyperplanes, there is a unique defining normal (up to scaling).
The set of bases of the matroid $X$ (\ie the subsequences of $X$ of cardinality $r$ and rank $r$) is denoted $\BB(X)$.
Let $x\in X$. If $x=0$, then $x$ is called a \emph{loop}.  If $\rank(X\setminus x)=r-1$, then $x$ is called a \emph{coloop}.

The set of flats of a given matroid $X$ ordered by inclusion forms a lattice (\ie a poset with joins and meets) called the \emph{lattice of flats} $\Lcal(X)$.
An \emph{upper set} $J\subseteq\Lcal(X)$ is an upward closed set, \ie $C\subseteq C'$, $C\in J$ implies $C'\in J$. 
We call $C\in \Lcal(X)$ a \emph{maximal missing flat}  if $C\not\in J$ and $C$ is maximal with this property.

The \emph{Tutte polynomial} \cite{brylawski-oxley-1992} 
\begin{align}
T_X(x,y) := \sum_{A\subseteq X}(x-1)^{r-\rank(A)} (y-1)^{\abs{A}- \rank(A)}
\end{align}
captures a lot of information about the matroid $\Mfrak(X)$.

\paragraph{Algebra}

$\sym(V)$ denotes the \emph{symmetric algebra} over $V$. This is a base-free version of the ring of polynomials over $V$.
The choice of a basis $B=\{b_1,\ldots, b_n\}\subseteq V$  yields an isomorphism  $\sym(V) \cong \K[b_1,\ldots, b_r]$.
For a definition and more background on algebra, see \cite{concini-procesi-book} or \cite{eisenbud-1995}.

A \emph{graded vector space} is a  vector space $V$  that decomposes into a direct sum $V=\bigoplus_{i\ge 0} V_i$.
 A graded linear map $f: V\to W$ preserves the grade, \ie $f(V_i)$ is contained in $W_i$. 
For a graded vector space, we define its \emph{Hilbert series} as the formal power 
series $\hilb(V,t):=\sum_{i\ge 0} \dim(V_i) t^i$. 
A \emph{graded algebra} $V$ has the additional property $V_i V_j \subseteq V_{i+j}$. 
We use the symmetric algebra with its natural grading. This grading is characterized by the property
that the degree one elements are exactly the ones that are contained in $V\setminus \{0\}$.
A $\Z^n$-\emph{multigraded ring} $R$ is defined similarly: $R$ decomposes into a direct sum $R= \bigoplus_{\vek a\in \Z^n}R_{\vek a}$
and
$R_\vek{a} R_\vek{b} \subseteq R_{\vek a + \vek b}$.

Note that a linear map $f: V\to W$ induces an algebra homomorphism $\sym(f) : \sym(V)\to \sym(W)$.

\smallskip

A \emph{derivation} on $\sym(V)$ is a $\K$-linear map $D$ satisfying Leibniz's law, \ie
$D(fg)=D(f)g+ fD(g)$ for $f,g\in \sym(V)$. For $v\in V=U^*$, we define the \emph{directional derivative} in direction $v$,
$D_v : \sym(U)\to \sym(U)$ as the unique derivation which satisfies $D_v(u)= v(u)$ for all $u\in U$.
For $\K=\R$ and $\K=\C$, this definition agrees with the analytic definition of the directional derivative. 
$\sym(V)$ can be identified with the ring of differential operators on $\sym(U)$.
Namely, $v_1\cdots v_k\in \sym(V)$ acts on $\sym(U)$ by mapping 
$f\in \sym(U)$ to $D_{v_1}(\ldots (D_{v_{k-1}}(D_{v_k}f))\ldots)$.

Now we define a pairing
$\pair{\cdot}{\cdot} : \sym(U) \times \sym(V) \to \K$. 
$\pair{f}{q}$ is defined to be the degree zero component 
of $q\cdot f$, where $q$ acts on $f$ as a differential operator.

\paragraph{Homogeneous ideals and their kernels}

An ideal $\Ical\subseteq \sym(V)$ is called a \emph{power ideal} \cite{ardila-postnikov-2009} if $\Ical= \ideal\{D_\eta^{e_\eta} : \eta \in Z \}$ for some $Z\subseteq V\setminus \{0\}$ and $e_\eta\in \N$.\footnote{%
In the original definition in \cite{ardila-postnikov-2009}, $1$ is added to every exponent $e_\eta$.
} 
$D_\eta$ denotes the image of $\eta$ under the canonical injection $V\hookrightarrow \sym(V)$.
By definition, power ideals are homogeneous.

\begin{Definition}
 Let $\Ical \subseteq \sym(V)$ be a homogeneous ideal. Its  \emph{kernel} (or inverse system 
 \cite{groebner-1938,abramson-2010,macaulay-1916}) 
  is defined as
 \begin{align}
   \ker \Ical &:= \{ f \in \sym(U) : \pair{f}{q} = 0 \text{ for all } q \in \Ical \} 
 \end{align}
\end{Definition}
Let $G$ be a set of generators for $\Ical$. 
Note that
 $\ker \Ical$ can also be written as $\{ f \in \sym(U) : g \cdot f=0 \text{ for all } g\in G \}$, where $g$ acts on $f$ as a differential operator.
It is known that for a homogeneous ideal $\Ical \subseteq \sym(V)$ of finite codimension 
the Hilbert series of $\ker \Ical$ and $\sym(V)/\Ical$ are equal.
For instance, this follows from  \cite[Theorem 5.4]{concini-procesi-book}.\footnote{%
$\ker \Ical$ is sometimes defined slightly differently in the literature:
first note that $\sym(U)$
($\approx$ polynomials) is a subspace of $\sym(V)^*$ ($\approx$ formal power series).
The pairing $\pair{\bullet}{\bullet}$ is  defined on $\sym(V)^* \times \sym(V)$
and $\ker \Ical$ is the subset of $\sym(V)^*$ that is annihilated by $\Ical$. 
It is then proven that if $\Ical$ has finite codimension, then $\ker \Ical$ is contained in $\sym(U)$, \ie in this case both definitions yield the same space.
}

\paragraph{Remark on the notation}

As zonotopal spaces were studied by people from different fields, the notation and the level of abstraction 
 used in the literature varies.
 Authors with a background in spline theory usually work over %
$\R^r$ and identify it with 
 its dual space via the canonical inner product. %
$\Pcal$-spaces and $\Ical$-ideals are then both subsets of the polynomial ring $\R[t_1,\ldots,t_r]$. Other authors
work in an abstract setting as we do. 
Since the Euclidean setting captures all the important parts of the theory,
a reader with no background in abstract algebra may safely make the following substitutions:
$\K=\R$, $U\cong V\cong \R^r$. %
$q\in \R[t_1,\ldots, t_r]=\sym(V)$ acts on $\R[t_1,\ldots, t_r]=\sym(U)$ as the differential operator that is obtained from $q$ by 
substituting $t_i \mapsto \diff{t_i}$.
The pairing can then be written as $\pair{f}{q}= (q\big(\diff{t_1},\ldots, \diff{t_r}\big) (f))(0)$.
We make those substitutions in Section~\ref{section:examples} (Examples).

Some authors work in the dual setting and consider a central hyperplane arrangement  instead of a finite sequence of vectors $X$.
While hierarchical zonotopal power ideals can be defined in both settings, it is natural for us to work with vectors as we are also interested in the zonotope $Z(X)$.

\section{Hierarchical zonotopal power ideals and their kernels}
\label{section:MainTheorem}

In this section, we define hierarchical zonotopal power ideals and show that their kernels have a nice description
as $\Pcal$-spaces.

The first subsection contains the definitions and the statement of the Main Theorem. In the second subsection, we prove some
simple facts and give explicit formulas for the $\Pcal$-spaces in two simple cases.
In the third subsection, we define deletion and contraction for pairs consisting of a matroid and an upper set in its lattice of flats. This is then used to give an inductive proof of the Main Theorem.

\subsection{Definitions and the Main Theorem}
Recall that $U=V^*$ denotes an $r$-dimensional vector space over a field $\K$ of characteristic zero and $X=(x_1, \ldots, x_N)$ denotes a finite sequence whose elements span $U$.

A vector $\eta \in V$ defines a flat $C\subseteq X$. Define $m_X(C)=m_X(\eta):=\card{X\setminus C}$. 
Sometimes, we write $m(C)$ instead of $m_X(C)$.
Given an upper set $J\subseteq \Lcal(X)$,  $\chi_J : \Lcal(X) \to \{0,1\}$ denotes its indicator function.
Again, the index is omitted if it is clear which upper set is meant.
$\chi$ can be extended to the power set of $X$ by %
$\chi(A):=\chi(\clos(A))$ for $A\subseteq X$.

For a given $x\in U$, we denote the image of $X$ under the canonical injection $U\hookrightarrow \sym(U)$ by $p_x$.
For $Y\subseteq X$, we define $p_Y:=\prod_{x\in Y} p_x$. For $\eta\in V$, we write $D_\eta$ for the image of $\eta$ under the canonical injection $V\hookrightarrow \sym(V)$ in order to stress the fact that we mostly think of $\sym(V)$ as the algebra 
generated by the directional derivatives on $\sym(U)$.
\begin{Definition}[Hierarchical zonotopal power ideals and $\Pcal$-spaces]
Let $\K$ be a field of characteristic zero, $V$ be a finite-dimensional $\K$-vector space of dimension $r\ge 1$ and $U=V^*$.
Let $X=(x_1,\ldots, x_N)\subseteq U$ be a finite sequence whose elements span $U$.
Let $k \ge -1$ be an integer and let $J \subseteq \Lcal(X)$ be a non-empty upper set, where $\Lcal(X)$ denotes
the lattice of flats of the matroid defined by $X$.

Let  $\chi : \Lcal(X)\to \{0,1\}$ denote the indicator function of $J$. 
Let $E : \Lcal(X) \to V$ be a normal selector function, \ie a function that assigns a defining normal to every flat.
Now define
\begin{align}
   \Ical'(X, k, J, E) &:= \ideal \Bigl\{D_{E(C)}^{m(C) + k + \chi(C)} : C \text{ hyperplane or %
maximal missing flat} \Bigr\} 
\nonumber
\\
  \Ical(X, k, J) &:= \ideal \left\{ D_{\eta}^{m(\eta)+k + \chi(\eta)} :  \eta \in V \setminus \{0\} \right\} \subseteq \sym(V) 
    \label{eq:DefinitionI} \displaybreak[2]   %
   \\
  \Pcal(X, k, J) &:= \spa S(X,k,J) \subseteq \sym(U) %
  \intertext{where}
   S(X,k,J) &:=  \{ fp_Y : Y\subseteq X, 0 \le \deg f \le \chi(X\setminus Y) + k -1 \}  
                             \quad\; \text{for } k\ge 1 
   \label{eq:DefinitionPgenerators} 
   \\
   S(X,0,J) &:= \{ p_Y : Y\subseteq X,\, \clos(X\setminus Y) \in J \}   \\
   S(X,-1,J) &:= \{ p_Y : \card{Y\setminus C} < m(C) -1 + \chi(C) \text{ for all } C\in \Lcal(X) \setminus \{X\} \}  
   \label{eq:DefinitionInternalPgenerators} 
\end{align}
\label{Definition:PIdefinition}
\end{Definition}
Note that the definition of $S(X,0,J)$ can be seen as a special case of the definition of $S(X,k,J)$ for $k\ge 1$. Therefore, we  distinguish
 only the two cases $k\ge 0$ and $k=-1$ in the proofs.

The condition $X\in J$ is only relevant in the case $k=0$. Then it ensures $1 \in S(X,0,J)$.
One can easily see that in the definition of $S(X,-1,J)$, it is sufficient to check only the inequalities associated to hyperplanes and to maximal missing flats.
If $x$ is a coloop and $X\setminus x \not\in J$, then $S(X,-1,J)=\emptyset$.

For examples, see Section~\ref{section:examples},  Remark~\ref{Prop:PspaceOneDim}, and Proposition~\ref{Prop:PspaceGeneralPosition}.

\begin{Theorem}[Main Theorem]
\sameterminologyasinmaindefinition
For $k=-1$, we assume in addition that $J\supseteq \Hcal$, 
\ie $J$  contains  all 
hyperplanes in $X$.
Then,
 \begin{align}
  \Pcal(X,k,J) &= \ker \Ical(X, k, J) \subseteq \ker \Ical'(X, k, J, E)
 \end{align}
Furthermore, for $k \in \{ -1, 0 \} $, $\Ical'(X,k,J,E)$ is independent of the choice of the normal selector 
 function $E$ and
\begin{align} 
 \Pcal(X,k,J) =  \ker \Ical(X,k,J)=\ker \Ical'(X, k, J, E)
\end{align}
\end{Theorem}
\begin{Remark}
Example~\ref{example:InternalAdditionalCondition}  explains why 
 there is an additional condition for $k=-1$ (see also Remark \ref{Remark:MainTheoremNotInternal}). 
Holtz, Ron, and Xu  \cite{holtz-ron-xu-2012} define  a different semi-internal structure. 
For a fixed $C_0\in \Lcal(X)$ and $J_{C_0}:=\{ C\in \Lcal(X) : C\supseteq C_0\}$,
 they show $\ker\Ical'(X,-1,J_{C_0})= \bigcap_{x\in C_0} \Pcal(X\setminus x,0,\{X\})$. However, 
they do not have a canonical generating set for this space. See Subsection~\ref{subsect:semiIntHilb}
for more details.
In the same paper, Holtz, Ron, and Xu define semi-external spaces that are the same as ours.
 However, they only identify them with the kernel of a power ideal in the special case where all maximal missing
 flats are hyperplanes.
\end{Remark}
From the Main Theorem and the results in Section \ref{section:HilbertSeries}, one can easily deduce the following two corollaries: 
\begin{Corollary}
In the setting of the Main Theorem, 
\begin{align}
\Pcal(X,k,\Lcal(X)) = \Pcal(X, k+1, \{X\})
\end{align}
\end{Corollary}
\begin{Corollary}
 The Hilbert series of $\Pcal(X,k,J)$ depends only on the matroid $\Mfrak(X)$, but not on the representation $X$.
 \label{Thm:IndependenceRepresentation}
\end{Corollary}
\begin{Remark}
One might wonder if similar theorems can be proven for  $k\le -2$. One would of course need to impose extra conditions on 
$X$ to ensure that the exponents appearing in the definition of the ideals are non-negative.
It is easy to see that $\Ical$ and $\Ical'$ are equal in this case (Lemma~\ref{Lemma:IgleichIprime}).
However, we do not know how to construct a generating set for their kernel. 
A different approach would be required: 
in general, their kernel is not spanned by a set of polynomials of type $p_Y$ for some $Y\subseteq X$
\cite{ardila-postnikov-2009}.%
\label{Rem:SmallerThanInternal}
\end{Remark}

\subsection{Basic results}
In this subsection, 
we prove three lemmas that are needed later on and we prove the Main Theorem in two special cases that are the base cases 
for the inductive proof in the next subsection. %

\begin{Lemma}
\label{Lemma:derivativesfaces}
Let $Y\subseteq X$ and let $\eta\in V$ be the defining normal for  $C\subseteq X$. Then,
 \begin{align}
   D_\eta p_Y = p_{Y\cap C} D_\eta p_{Y\setminus C} 
 \end{align}
\end{Lemma}
\begin{proof}
 This is a direct consequence of Leibniz's law.
\end{proof}

\begin{Lemma}
 Let $u_1,\ldots, u_s\in U$ and let $k\in \N$. Then,
 \label{Lemma:powersOfSums}
 \begin{align*}
  \spa \{ (\alpha_1 u_1 +\ldots + \alpha_s u_s)^k : \alpha_i \in \K \setminus \{0\}\} = \spa \left\{ \prod_{i=1}^s u_i^{a_i} : \sum_{i=1}^s a_i=k,\, a_i\in \N \right\}
 \end{align*}
\end{Lemma}
\begin{proof}
 ``$\subseteq$'' is clear. %
Let $L$ denote the number of monomials of the 
form %
$\prod_{i=1}^s u_i^{a_i}$
($\sum_i a_i=k$, $a_i \in \N$). Order those monomials lexicographically. 
By induction, 
 we can see that there are polynomials $p_1,\ldots, p_L$ contained in the set on the left \st the leading term of $p_i$ is the $i$th monomial. This implies that all those monomials are contained in the set on the left side.
\end{proof}

\begin{Lemma}
\label{Prop:MainTheoremEasyInclusion}
    $\Pcal(X,k,J) \subseteq \ker \Ical(X,k,J) \subseteq \ker \Ical'(X,k,J)$ 
holds for all $k\ge -1$ and all $J\subseteq \Lcal(X)$. 
\end{Lemma}
\begin{proof}
 The second inclusion is clear.
 For the first, we generalize the proof of  \cite[Theorem 3.5]{holtz-ron-2011}:
 it suffices to prove that every generator of $\Ical(X,k,J)$ annihilates every element of $S(X,k,J)$. 
For $k=-1$, this is obvious. Now consider the case $k\ge 0$.
Let $C$ be a flat, $\eta$ a defining normal for $C$, $Y\subseteq X$, $\deg f\le k + \chi(X\setminus Y)-1$.
Set $e(C):=m(C) + k + \chi(C)$. By Lemma~\ref{Lemma:derivativesfaces},
\begin{align}
   D_\eta^{e(C)} fp_Y %
   & \gleich{\hspace{1mm}} p_{Y\cap C}\sum_{i=0}^{e(C)} \binom{e(C)}{i}  D_\eta^{i} f D_\eta^{e(C)-i} p_{Y\setminus C} \\
   & \gleich{(*)}
    p_{Y\cap C}\sum_{i=k+ \chi(C)}^{k + \chi(X\setminus Y)-1} \binom{e(C)}{i}  D_\eta^{i} f D_\eta^{e(C)-i} p_{Y\setminus C} \label{equation:MainTheoremEasyInclusion} 
\end{align}
$(*)$ holds because $f$ does not survive $k+\chi(X\setminus Y)$ differentiations and $p_{Y\setminus C}$ is annihilated by $m(C)+1$ differentiations. 
Suppose the term in \eqref{equation:MainTheoremEasyInclusion} is not zero. Then $\chi(X\setminus Y)=1$ and $\chi(C)=0$. Furthermore, 
$m(C)$ differentiations in direction $\eta$ do not annihilate $p_{Y\setminus C}$. This is only possible if $Y\setminus C=X\setminus C$.
This implies $X\setminus Y\subseteq C$. Then $\chi(X\setminus Y)\le \chi(C)$. This is a contradiction.
\end{proof}
Now we give explicit formulas for $\Pcal(X,k,J)$ and $\Ical(X,k,J)$ in two particularly simple cases:

\begin{Remark}
 Suppose that $\dim U = 1$ and that $X$ contains $N'$ non-zero entries. Let $x\in U$ and $y\in V$ be non-zero vectors. 
 Note that $\clos(\emptyset)$ is the only hyperplane in $X$.
 Hence,
 $\Ical'(X,k,J) = \Ical(X,k,J) = \ideal \{ D_y^{N' + k + \chi(\emptyset)}  \}$ and $\Pcal(X,k,J) = \spa \{ p_x^i : 
 i\in \{0,1,\ldots, N' - 1 + k + \chi(\emptyset) \}\}$. 
 \label{Prop:PspaceOneDim}
\end{Remark}

\begin{Proposition}
 \label{Prop:PspaceGeneralPosition}
  \sameterminologyasinmaindefinition
  Let $X=(x_1,\ldots,x_r)$ be a basis for $U$. Let $(y_1,\ldots, y_r)$ denote the dual basis of $V$.
  Then, $\Pcal(X,k,J) = \ker \Ical(X,k,J) \subseteq \ker \Ical'(X,k,J,E)$.  %

Furthermore, for $k\in\{-1,0\}$, $\ker \Ical(X,k,J) = \ker \Ical'(X,k,J,E)$ for all normal selector functions $E$.
More precisely,
  writing $p_i:=p_{x_i}$ and $D_i:=D_{y_i}$ as shorthand notation, we get

\begin{align}
 \Ical(X,k,J) & = \ideal \left\{  \prod_{i \in I}D_{i}^{a_i+1}    : %
  \sum_{i\in I} a_i = k +
 \chi( X \setminus \{ x_i : i\in I\}) \right\} 
 \\
 \Pcal(X,k,J) &=  \spa \Biggl \{ \prod_{i\in I}p_{i}^{a_i+1} :  %
  \sum_{i \in I} a_i \le  k + \chi( X \setminus \{x_i : i\in I\}  ) -1 \Biggr\}
 \label{eq:PspaceGeneralPosition}
\end{align}
where $I\subseteq [r]$ and $a_i\in \N$.\\
For $k=0$, this specializes to 
 $\Pcal(X,0,J) =  \spa \left \{ p_Y :   X \setminus Y  \in J \right\}$.\\
For $k=-1$, %
$\Ical(X,-1,J)=\ideal \{D_1,\ldots, D_r \}$ if $J \supseteq \Hcal$ 
and $\Ical(X,-1,J)=\ideal\{1\}$ otherwise.
\end{Proposition}
\noindent
For a two-dimensional example of this construction, see Example~\ref{Example:coloopRecursion}
 and Figure~\ref{fig:coloopRecursion}.

\begin{proof}
This proof generalizes the  proof of Proposition 4.3 in \cite{ardila-postnikov-2009}.
The statements about $k=-1$ are trivial.

Every flat of $X$ can be written as $C=X\setminus \{ x_i : i \in I  \}$ for some $I\subseteq [r]$.  %
The set of defining normals for $C$ is given by $\bigl\{ \sum_{i\in I} \alpha_i y_i : \alpha_i \in \K \setminus \{0 \} \bigr\}$.

First, we  show that for $k=0$, $\Ical(X,k,J)=\Ical'(X,k,J,E)$.
$y_i$ is the defining normal for the hyperplane $X\setminus x_i$. Hence,
 \begin{align}
D_i^{1+\chi(X\setminus x_i)} \in \Ical'(X,k,J,E) \qquad \text{for } i=1,\ldots,r.
\label{eq:HyperplaneGenerator}
\end{align}
Fix a flat $C$.
Let $D_{\eta_C}^{m(C)+\chi(C)}$ be a generator of $\Ical(X,k,J)$. We prove now that 
$D_{\eta_C}^{m(C)+\chi(C)}$ is contained in $\Ical'(X,k,J,E)$.
\begin{asparaenum}[\itshape C{a}se 1:]
 \item $C\in J$. Hence, $D_{\eta_C}^{m(C)+\chi(C)} = ( \sum_{i\in I}\alpha_i D_{i})^{\card I + 1}$ 
 for some $I\subseteq [r]$ and  $\alpha_i\in \K\setminus \{0\}$. In the monomial expansion of this term, every monomial contains a square. 
 By \eqref{eq:HyperplaneGenerator}, all squares are contained in $\Ical'(X,k,J,E)$.
 \item $C \not\in J$. Let $C'$ be a maximal missing flat that contains $C$. 
Then, 
\begin{align}
 D_{E(C')}^{m(C')+\chi(C')} = \bigl( \sum_{x_i \not \in C'} \lambda_i D_i\bigr)^{m(C')} \in \Ical'(X,k,J,E).
\end{align}
 In the monomial expansion of this polynomial, there is only one monomial that does not contain a square: $q:= \prod_{x_i\not \in C'} D_i$. %
 It follows from the definition of $\Ical'(X,k,J,E)$ and \eqref{eq:HyperplaneGenerator} that $q\in \Ical'(X,k,J,E)$. In the monomial expansion of $D_{\eta_{C}}^{m(C)+\chi(C)}$, there are only monomials containing squares and a monomial that is a multiple of $q$. Hence, 
$D_{\eta_C}^{m(C)+\chi(C)} \in \Ical'(X,k,J,E)$. 
\end{asparaenum}

\bigskip

One can easily see that 
\eqref{eq:PspaceGeneralPosition} describes the $\Pcal$-space
by comparing \eqref{eq:PspaceGeneralPosition} with \eqref{eq:DefinitionPgenerators} and taking into account that 
$X$ is a basis for $U$.

Using Lemma~\ref{Lemma:powersOfSums}, we can calculate $\Ical(X,k,J)$:
\begin{align*}
 \Ical(X,k,J) & = \ideal  \left\{ \left(\sum_{i\in I} \alpha_i D_{i}\right)^{ \card{I} + k + \chi( X \setminus \{x_i : i\in I\})}
 : I\subseteq [r], \alpha_i \in \K\setminus \{0\} 
  \right\}    \\
  & = 
 \ideal \left\{  \prod_{i\in I} D_{i}^{a_i+1}   : I\subseteq [r], a_i \in \N, \sum_{i\in I} a_i=  k +
 \chi( X \setminus \{ x_{i} : i \in I \}) 
\right\}  
\end{align*}
It is now clear that $\ker \Ical(X,k,J)=\Pcal(X,k,J)$.
\end{proof}
The following Lemma implies $\Ical(X,k,J)=\Ical'(X,k,J,E)$ for $k\le 0$,
using the Main Theorem for $k=0$ as base case
(cf. Remark~\ref{Rem:SmallerThanInternal}).
\begin{Lemma}
\label{Lemma:IgleichIprime}
 Let $J\subseteq \Lcal(X)$ be an arbitrary upper set 
 and $k$ be an arbitrary integer. 
 If $\Ical(X,k,J)$ is contained in $\sym(V)$ (\ie $m(C)\ge k$ for all flats  $C$) and $\Ical'(X,k,J,E)=\Ical(X,k,J)$ 
for all normal selector functions $E$,
then
 $\Ical'(X,l,J,E)=\Ical(X,l,J)$ for all $l\le k$ which satisfy $\Ical(X,l,J)\subseteq \sym(V)$ and all normal selector functions $E$.
\end{Lemma}

\begin{proof}
 Let $D_\eta^{m(\eta) + l + \chi(\eta)}$ be a generator of $\Ical(X,l,J)$. We show that this generator is contained in 
 $\Ical'(X,l,J,E)$.
By induction, we may suppose that $D_\eta^{m(\eta) + l+1 + \chi(\eta)}\in \Ical'(X,l+1,J)$, \ie there exist
$q_i\in \sym(V)$ and  $D_{\eta_i}^{m(\eta_i) + l + 1 + \chi(\eta_i)}$ generators of $\Ical'(X,l+1,J)$ \st
\begin{align}
D_\eta^{m(\eta) + l + 1 + \chi(\eta)} = \sum_{i} q_i D_{\eta_i}^{m(\eta_i) + l + 1 + \chi(\eta_i)} 
\label{eq:IgleichIprimeInd}
\end{align}
Let $u\in U$ be a vector \st $\eta(u)=1$. We consider $p_u$ as a differential operator on $\sym(V)$. By applying 
$p_u$ to \eqref{eq:IgleichIprimeInd}, 
we see that $D_\eta^{m(\eta) + l + \chi(\eta)}$ is contained in $\Ical'(X,l,J,E)$.
\end{proof}

\subsection{Deletion and Contraction}
\label{subsection:delcon}
In the third paragraph of this subsection, we  prove the Main theorem. 
The proof is inductive using deletion and contraction. 
In the first paragraph, we define those two operations for representable matroids.
In the second paragraph, we  define them for upper sets.

\subsubsection{Matroids under deletion and contraction}

Two important constructions for matroids are deletion and contraction of an element.
Let $x \in X$. The \emph{deletion} of $x$ is the matroid defined by the sequence $X\setminus x$. 

For the rest of this paragraph, fix an element $x\in X$ that is not a loop. Let $\pi_x : U \to U/x$ denote the projection to the quotient space.  The \emph{contraction} of $x$ is the matroid defined by the sequence $X/x$ which contains 
the images of the elements of $X\setminus x$ under $\pi_x$.

We want to be able to see $\sym(U/x)$ as a subspace of $\sym(U)$. For that, pick a basis 
 $B=\{b_1,\ldots, b_{r}\}\subseteq U$ with $b_r=x$.
Let $W:=\spa\{b_1,\ldots, b_{r-1}\}$. Then we have an isomorphism $U/x\cong W$ which extends to an isomorphism 
$\sym(U/x)\cong \sym(W)\subseteq \sym(U)$. Under this identification, $\sym(\pi_x)$ becomes the map that sends
$p_x$ to zero and maps all other basis vectors to themselves. Then $\sym(U)\cong \sym(W)\oplus p_x \sym(U)$. 

Let $Y\subseteq X\setminus x$. We  write $\bar Y$ to denote the subsequence of $X/x$ with the same index set as $Y$ and vice versa.
 Let $\bar C\subseteq X/x \subseteq W$ be a flat and $\eta \in W^*$
 be a defining normal for $\bar C$.
 Since $W^*\cong x^\perp:= \{ v \in V : v(x) = 0\}$, $\eta$ is also a defining normal for the flat $C\cup \{x\}\subseteq X$.

\subsubsection{The lattice of flats under deletion and contraction}

In this paragraph, we discuss how the lattice of flats of a matroid behaves under deletion and contraction and for a given upper set $J$ we define upper sets $J\setminus x \subseteq \Lcal(X\setminus x)$ and $J/x \subseteq \Lcal(X/x)$.

For the whole paragraph, fix an element $x\in X$ that is not a loop.
First, we exhibit some relations between the lattices of flats of $X$, $X\setminus x$ and $X/x$.
There are two bijective maps:
\begin{align}
 L_x : \Lcal(X\setminus x)  & \to  %
 \{ C \in \Lcal(X) : C = \clos(C\setminus x) \}  \\ %
 L^x :  \Lcal(X / x)  & \to  \{ C \in \Lcal(X) : x\in C  \} 
\end{align}
The maps are given by $L_x(C) := \clos\nolimits_X(C)$, $L^{-1}_x(C) :=C\setminus x$,  $L^x(\bar C) := C\cup x$ and
$(L^x)^{-1}(C) :=\overline{C\setminus x}$.

\begin{Definition}
Let $J\subseteq \Lcal(X)$ be an upper set. Then define
\begin{align*}
J\setminus x &:= \{ C \setminus x : C\in J \text{ and } C=\clos(C\setminus x)
\} 
=  L_x^{-1}( J \cap L_x(\Lcal(X \setminus x))) \subseteq \Lcal(X\setminus x) \\
J/x &:= \{ (\overline{C \setminus x}) : x\in C \in J  \} 
= (L^x)^{-1}(J \cap L^x(\Lcal(X/x))) \subseteq \Lcal(X / x)
\end{align*}
\end{Definition}
It is easy to check that those two sets are upper sets.
The following statement on the indicator functions is also easy to prove:
\begin{Lemma}
Let $x\not\in Y\subseteq X$. Then
$\chi_{J\setminus x}(Y)=\chi_J(Y)$ and 
 $\chi_{J/x}(\bar Y)= \chi_J(Y\cup x)$.
\label{Lemma:CharFuncDelCon}
\end{Lemma}
From this, we can deduce the following:
\begin{Lemma}
\begin{compactenum}[(1)]
\item If $C\subseteq X\setminus x$ is a maximal missing flat for $J\setminus x$ then $C$ or $C\cup x$ is a maximal missing flat for $J$.
\item 
If $\bar C\subseteq X / x$ is a maximal missing flat for $J/ x$ then $C\cup x$ is a maximal missing flat for $J$.
\end{compactenum}
\label{Lemma:MaxMissingFlatDelCon}
\end{Lemma}
We also need the following two facts:
\begin{Remark}
 Let $x\in X$ be neither a loop nor a coloop.
 Suppose that $J\supseteq \Hcal(X) $.
Then,
\begin{compactenum}[(1)]
 \item    $J\setminus x \supseteq \Hcal(X\setminus x)$.
  This follows from the fact that $\Lcal(X\setminus x)$ contains exactly the flats $C$ that satisfy $\rank(C)=\rank(C\setminus x)$.
 \item   $J / x \supseteq \Hcal(X/x)$.
  This follows from the fact that $\Lcal(X/x)$ contains exactly the flats containing $x$ and 
 the fact that contraction reduces the rank of a flat containing $x$ by one.
\end{compactenum}
\label{Rem:InternalIdealDelCon}
\end{Remark}

\subsubsection{Proof of the Main Theorem}

In this paragraph, we prove the Main Theorem. The following proposition is a side product of the deletion-contraction proof:
\begin{Proposition}
\sameterminologyasinmaindefinition
Suppose that $x\in X$ is neither a loop nor a coloop. %
For $k=-1$, we assume in addition that $J \supseteq  \Hcal$  %
or $J=\{X\}$, \ie $J$  contains either all or no 
hyperplanes in $X$.

Then the following is an exact sequence of  graded vector spaces:
 
\begin{align}
 0 \to \ker\Ical(X\setminus x, k, J\setminus x) (-1) \stackrel{\cdot p_x}{\longrightarrow} \ker\Ical(X,k,J) 
  \stackrel{\sym(\pi_x)}{\longrightarrow} \ker\Ical( X / x, k, J/x) \to 0
\end{align}

If $x\in X$ is a loop, then $\ker \Ical(X \setminus x, k, J \setminus x)=\ker \Ical(X, k, J)$.
For $k \in \{ -1, 0 \}$, both statements also hold if we replace $\Ical$ by $\Ical'$.

Here, $(\cdot)(-1)$ denotes the graded vector space $(\cdot)$ with the degree shifted up by one and $\sym(\pi_x)$ denotes the algebra homomorphism that maps $p_v$ to $p_{\pi_x(v)}$.
\label{Prop:ExactSequence}
\end{Proposition}

The proof of this proposition is inductive. It uses the following lemma:

\begin{Lemma}
\label{Lemma:HilfslemmaHauptsatz}
Suppose that we are in the same setting as in Proposition~\ref{Prop:ExactSequence}.
Let $x\in X$ be neither a loop nor a coloop. 
Suppose that $\Pcal(X\setminus x,k,J\setminus x) = \ker\Ical(X\setminus x,k,J\setminus x)$ and
 $\Pcal(X / x,k,J / x) = \ker\Ical(X / x,k,J / x)$. 
\begin{asparaenum}[(1)]
 \item Then, the following sequence is exact:
 \label{enum:ExactSequencePropPartA}
 \begin{align}
   0 \to \ker \Ical(X\setminus x, k, J \setminus x) (-1) \stackrel{\cdot p_x}{\longrightarrow} \ker \Ical(X,k,J) 
 \stackrel{ \sym(\pi_x) }{\longrightarrow} \ker\Ical( X / x,k, J/x) \to 0 
 \end{align}
 \item If we suppose in addition that 
 \label{enum:ExactSequencePropPartB}
 $\Ical'(X\setminus x,k,J\setminus x,E')=\Ical(X\setminus x,k,J\setminus x)$ and 
 $ \Ical'(X / x,k,J / x,E'') = \Ical(X / x,k,J / x) $, 
%
 the following sequence is exact:
\begin{align}
 0 \stackrel{ \hspace*{7ex} }{\longrightarrow} &\ker \Ical'(X\setminus x,k,J\setminus x,E') (-1) \stackrel{\cdot p_x}{\longrightarrow} \ker \Ical'(X,k,J,E) \nonumber\\ 
 \stackrel{ \sym(\pi_x) }{\longrightarrow} &\ker\Ical'( X / x, k, J/x, E'') \longrightarrow 0 
 \label{eq:exactSequenceLemma} 
\end{align}
 Here, 
 $E'$ and $E''$ denote the restrictions of 
 $E$ to $\Lcal(X\setminus x)$ and $\Lcal(X/x)$, \ie
 $E'(C)= E(\clos(C))$ and $E''(\bar C) =  E(C\cup x)$.
\end{asparaenum}
\end{Lemma}

\begin{proof}[Proof of Lemma~\ref{Lemma:HilfslemmaHauptsatz}]
We only prove part \eqref{enum:ExactSequencePropPartB}. The reader will notice, that the same proof 
with some obvious modifications can be used to prove 
part \eqref{enum:ExactSequencePropPartA}, unless 
$k=-1$ and $J=\{X\}$. 
In that case, both parts are equivalent 
 by 
Lemma~\ref{Lemma:IgleichIprime}.

Before starting with the proof, we introduce some additional  notation, which is only used here.
For a flat $C$ defined by $\eta$, we write $e_X(\eta)=e_X(C) := m_X(C) + k + \chi(C)$.
As described above, we fix a subspace $W\subseteq U$ complementary to $\spa x$ and identify $U/x$ with $W$.
 Hence, $\ker \Ical'(X/x,k,J/x,E'')\subseteq \sym(W)$.

Let $f\in \sym(U)$ and $v,\eta \in V$. 
Let $t$ be a formal symbol. Then $f(v +t\eta)\in \K[t]$ and the following Taylor expansion formula holds: $f(v +t\eta)=\sum_{k\ge 0} \frac{D_\eta^k}{k!} f(v)t^k$. Now define $\rho_f : V\to \N$, the \emph{directional degree function} of $f$ \cite{ardila-postnikov-2009}, as the function 
 which assigns to $\eta$ the degree of the univariate polynomial $f(v +t\eta) \in \K[t]$ for generic $v$. %
We obtain
$\rho_{fg}=\rho_f + \rho_g$ by comparing %
the Taylor expansion of $f\cdot g$ with %
the product of the Taylor expansions of $f$ and $g$. 
$\rho_f(\eta)$ tells us how many derivations $f$ survives in direction $\eta$. Hence, $\rho$ can be used to describe $\ker \Ical(X,k,J)$ and $\ker \Ical'(X,k,J,E)$. Namely,
\begin{align}
 \ker \Ical(X,k,J)&=\{ f\in \sym(U) : \rho_f(\eta) < e(\eta),\, \eta\in V\setminus \{0\} 
 \}%
\end{align}

\smallskip

\noindent Now we come to the main part of the proof.  It is split into five parts:

\newcommand{\paraenumIndent}{3mm}
\begin{asparaenum}[\itshape (i)]
\item \emph{$\cdot p_x$  is well-defined, \ie really maps to $\ker \Ical'(X,k,J,E)$:} 
due to Lemma~\ref{Prop:MainTheoremEasyInclusion}, it suffices to prove
$ S(X \setminus x,k,J\setminus x) \stackrel{\cdot p_x}{\hookrightarrow} S(X, k, J)$.
 For $k\ge 0$, this follows directly from Lemma~\ref{Lemma:CharFuncDelCon}.
 For $k = -1$, consider 
 $p_Y\in S(X\setminus x, -1, J\setminus x)$ and
$C\in \Lcal(X)$. Then $C\setminus x \in \Lcal(X\setminus x)$ and 
 $\chi_{J\setminus x}(C\setminus x) \le \chi_J(C)$. 
 One can easily deduce 
 $\card{(Y \cup  x )\setminus C} < m(C) - 1 + \chi_J(C)$ 
from the corresponding inequality for $C\setminus x$. 
\label{enum:KeyLemmaA}
\item 
\emph{$\sym(\pi_x)$ is well-defined:} 
let $g \in \ker \Ical'(X,k,J,E)$ and let $h:=(\sym(\pi_x))(g)$. 
Let $\bar C \in \Lcal(X/x)$ be a maximal missing flat or a hyperplane, respectively.
By Lemma~\ref{Lemma:MaxMissingFlatDelCon}, $C\cup x \in \Lcal(X)$ is a maximal missing flat or a hyperplane, respectively. 
 Let $\eta:=E(C \cup x)$. 
This implies 
$E''(\bar C)=\eta$.  We need to prove $\rho_h(\eta) < {e_{X/x}(\bar C)}$.

Note that $m_{X/x}(\bar C)=m_X(C\cup x)$ and by Lemma~\ref{Lemma:CharFuncDelCon}, $\chi_{J/x}(\bar C) = \chi_J(C\cup x)$ 
Hence, $e_{X/x}(\bar C)=e_X(C \cup x)$.
$g$ can be uniquely written as $ g=h + p_x g_1 $ for some $g_1\in \sym(U)$.
For all $k\in \N$, $D_\eta^k g = D_\eta^kh + p_x D_\eta^k g_1$.
As $p_x$ does not divide $h$, this implies $\rho_h(\eta)\le \rho_g(\eta)$. In summary, we get
\begin{align}
 e_{X/x}(\bar C)=e_X(C \cup x) > \rho_g(\eta) \ge \rho_h(\eta)
\end{align}
\item
\emph{Injectivity of $\cdot p_x$:} clear.
\item
\emph{Exactness in the middle:}
let $g\in \ker\Ical'(X,k,J,E)$ and $\sym(\pi_x)(g)=0$. 
This implies that $g$ can be written as $g=p_xh$ for some $h\in \sym(U)$.
We need to show that %
$h\in \ker\Ical'(X\setminus x,k,J\setminus x,E')=\ker\Ical(X\setminus x,k,J\setminus x)$.

Let $C$ be a maximal missing flat (resp.\ hyperplane) in $X\setminus x$. 
By Lemma~\ref{Lemma:MaxMissingFlatDelCon},
$C'=C$ or $C'=C\cup x$ is a maximal missing flat 
 (resp.\ hyperplane) in $X$. 
Let $\eta:=E(C')$. $\eta$ is also a defining normal for $C\subseteq X\setminus x$. 
By definition of $E'$, 
$\eta=E'(C)$.
We show now that $\rho_h(\eta) < e_{X\setminus x}(C)$. 

If $x\in C'$, then $\rho_{p_x}=0$, $m_{X\setminus x}(C)=m_X(C')$, and $\chi_{J\setminus x}(C) =  
\chi_J(C')$.
If $x\not \in C'$, then $\rho_{p_x}(\eta)=1$, $m_{X\setminus x}(C) + 1 =m_X(C')$, and $\chi_{J\setminus x}(C)=\chi_J(C')$. 
So, in both cases, $e_{X\setminus x}(\eta) + \rho_{p_x}(\eta) = e_X(\eta)$.
This implies
\begin{align}
 e_{X\setminus x}(\eta) =  e_X(\eta) - \rho_{p_x}(\eta) >  \rho_{p_xh}(\eta) - \rho_{p_x}(\eta) = \rho_{h}(\eta)
\end{align}
\item
\emph{Surjectivity of $\sym(\pi_x)$:} 
we consider the case $k\ge 0$ first.
Let $fp_{\bar Y} \in S(X / x,k,J / x)$.
It suffices to prove that
$fp_Y \in S(X,k,J)$. 
Since $x \not\in Y$, by Lemma~\ref{Lemma:CharFuncDelCon}, $\chi_{J/x}((X/x) \setminus \bar Y )
 = \chi_J(X \setminus Y)$. This implies $fp_Y \in S(X,k,J)$.
\smallskip

Now consider the case $k=-1$. This requires a little more work. There are two subcases:

{ 
 \item[(a)] \emph{$J \supseteq \Hcal$:} 
let $p_{\bar Y} \in S(X/x, -1, J/x)$. We show now that $p_Y \in S(X,-1,J)$. 
Let $C \in \Lcal(X) \setminus \{X\}$.
Suppose first that $x\in C$ or $\codim C \ge 2$. Then $D:=\clos(C\cup x) \neq X$.
By assumption, $\card{\bar Y\setminus (\overline{D\setminus x})} < m_{X/x}(\overline{D\setminus x}) -1 + \chi_{J\setminus x}(\overline{D\setminus x})$. 
By Lemma~\ref{Lemma:CharFuncDelCon}, this implies
\begin{align}
\card{Y\setminus D} &<  m_X(D) -1 + \chi_J(D)
\label{eq:InternalSurjectivityProof} 
\end{align}
Since $x\in D\setminus C$ and $x\not\in Y$, 
$\card{Y\setminus C}- \card{Y\setminus D} \le m_X(C) - m_X(D) - 1$.
Adding this inequality to \eqref{eq:InternalSurjectivityProof}, we obtain the desired inequality: 
\begin{align}
\card{Y\setminus C} &<  m_X(C) -1 + \chi_J(C) \label{equation:internalDefiningRelation}
\end{align}

  Now suppose that $C$ is a hyperplane and $x\not\in C$. By assumption, $\chi_J(C)=1$.
\indent Since  $x\not\in Y\cup C$, we can deduce that \eqref{equation:internalDefiningRelation} holds for this $C$.
}

 \item[(b)] \emph{$J= \{X\}$:}
This can be shown by a dimension argument using the fact that the dimension of
$\Pcal(X-1,\{X\})$ equals the cardinality of the set of so called internal bases $\BB_-(X)$
  \cite[Theorem 5.9]{holtz-ron-2011}.
If $x\in X$ is the minimal element, the following deletion-contraction equality holds: 
\begin{align}
\abs{\BB_-(X)} = \abs{\BB_-(X \setminus x)} + \abs{\BB_-(X/x)}
\end{align}
\end{asparaenum}
\end{proof}

\begin{proof}[Proof of Proposition~\ref{Prop:ExactSequence} and of the Main Theorem]
We generalize the proof of \cite[Propositions 4.4 and 4.5]{ardila-postnikov-2009}.

Loops can safely be ignored: they are contained in every flat $C$, thus 
$m_X(C)=m_{X\setminus x}(C)$ and $\Lcal(X)\cong \Lcal(X\setminus x)$ if $x$ is a loop. 
From now on, we suppose that $X$ does not contain loops.

We prove both statements by induction on the number of elements of 
$X$ that are not coloops.
The reader should check that our reasoning below also works for $k=-1$, although in that case, $\Pcal$-spaces might be zero.
Remark~\ref{Rem:InternalIdealDelCon} ensures that an upper set that contains 
all hyperplanes
 preserves this structure under deletion and contraction.

If $X$ contains only coloops the Main Theorem follows from Proposition~\ref{Prop:PspaceGeneralPosition}.
Now suppose that $x\in X$ is not a coloop and that the Main Theorem holds for $X/x$ and $X\setminus x$.
In addition, we assume $\dim U \ge 2$. If $\dim U= 1$, the statement follows from Remark~\ref{Prop:PspaceOneDim}.

By Lemma~\ref{Lemma:HilfslemmaHauptsatz}, the following
sequence is exact:
 \begin{align*}
   0 \to \ker \Ical(X\setminus x, k, J \setminus x) (-1) \stackrel{\cdot p_x}{\longrightarrow} \ker \Ical(X,k,J) 
 \stackrel{ \sym(\pi_x) }{\longrightarrow} \ker\Ical( X / x,k, J/x) \to 0 
 \end{align*}
Every short exact sequence of vector spaces splits. Hence, 
$\ker \Ical(X,k,J) = p_x \cdot \ker \Ical(X\setminus x, k, J \setminus x) \oplus 
 \ker\Ical( X / x,k, J/x)$. For $k\in \{-1,0\}$, the same argumention also works for $\Ical'(X,k,J,E)$.
To conclude, we recall the following two statements that were shown in the proof of 
Lemma~\ref{Lemma:HilfslemmaHauptsatz}:
 (i)~$p_x \cdot S(X\setminus x, k, J\setminus x) \subseteq S(X, k, J)$ 
 and (ii)~$\sym(\pi_x) : S(X, k, J) \to S(X/x, k, J/x)$ is surjective, if $(k,J)\neq (-1,\{X\})$.
\end{proof}

\begin{Remark}
\label{Remark:InternalMainTheoremProblems}
The map  $\sym(\pi_x) : S(X, -1, \{X\}) \to S(X/x, -1, \{X/x\})$ is in general not surjective
(cf. Example~\ref{Example:InternalRecursion}).
Proposition \ref{Prop:ExactSequence} is false for arbitrary $J\not \supseteq \Hcal$ 
(cf. Example~\ref{example:InternalAdditionalCondition}).
The difficulty of the case $k=-1$ was already observed by Holtz and Ron. They conjectured that the Main Theorem
holds in the internal case, \ie for $k=-1$ and $J=\{X\}$ 
 \cite[Conjecture 6.1]{holtz-ron-2011}.
An  incorrect ``proof'' of this conjecture appeared in \cite{ardila-postnikov-2009}, which assumed that
 $\sym(\pi_x) : S(X, -1, \{X\}) \to S(X/x, -1, \{X/x\})$ is always surjective.
 The same authors later found a counterexample for the conjecture \cite{ardila-errata-2012}.
\label{Remark:MainTheoremNotInternal}
\end{Remark}

\section{Bases for $\Pcal$-spaces}
\label{section:basis}
In this section, we show how a basis for $\Pcal(X,k,J)$ can be selected from $S(X,k,J)$ for $k\ge 0$. 
Our construction depends on the order on $X$. This order is used to define the notions of internal and external activity
(see \cite[Section 6.6.]{brylawski-oxley-1992} for a reference). 
Our result is a generalization of \cite[Proposition 4.21]{ardila-postnikov-2009} to hierarchical spaces.
At the end of this section, there is a remark on the case $k=-1$.
\smallskip

Recall that  $\BB(X)$ denotes the set of all bases $B\subseteq X$. Fix a basis $B\in \BB(X)$. 
$b\in B$ is called \emph{internally active} if $b=\max (X\setminus \clos(B\setminus b))$,
\ie $b$ is the maximal element of the unique cocircuit %
contained in $(X\setminus B)\cup b$.
The set of internally active  elements in $B$ is denoted  $I(B)$. 
$x\in X\setminus B$ is called \emph{externally active} if 
$x\in \clos \{ b \in B : b \le x \}$,
\ie $x$ is the maximal element of the unique circuit contained in $B\cup x$.
The set of externally active elements with respect to $B$ is denoted $E(B)$.

\begin{Definition}
 \sameterminologyasinmaindefinition
 In addition, let $k\ge 0$. Then define
\begin{align}
  \Gamma(X,k,J) &:= \bigl\{ (B,I,\vek a_I) :  B \in \BB(X),\, I\subseteq I(B),\,  \vek a_I \in \N^I, \nonumber\\
                    &\qquad \qquad\qquad\qquad \sum_{x\in I} a_x \le k +  \chi((B\cup 
                     E(B))\setminus I) -1 \bigr\} 
	\label{eq:DefinitionGamma}
	 \\
  \Bcal(X,k,J) &:= \Bigl\{ p_{X\setminus (B \cup E(B))} \prod\limits_{x\in I} p_x^{a_x+1} : (B,I,\vek a_I) \in \Gamma(X,k,J) \Bigr \} \subseteq \sym(U)   
  \label{eq:DefinitionPspaceBasis} 
\intertext{For $k=0$, this specializes to}
 \Bcal(X,0,J) &= \left\{ p_{(X\setminus (B \cup E(B)))\cup I}  : B \in \BB(X),\, I\subseteq I(B),\,\clos((B\cup E(B))\setminus I) \in J \right\} \nonumber
 \end{align}
\label{Definition:BasesForPspaces}
\end{Definition}

Note that \textit{a priori}, it is unclear whether the set $\Gamma(X,k,J)$ has the same cardinality as the set $\Bcal(X,k,J)$ since we do not know if  distinct elements of $\Gamma(X,k,J)$ correspond to distinct polynomials in $\Bcal(X,k,J)$.
This desired property only becomes clear in the proof of the following theorem:

\begin{Theorem}[Basis Theorem]
 \sameterminologyasinmaindefinition
 In addition, let $k\ge 0$.
Then $\Bcal(X,k,J)$ is a basis for $\Pcal(X,k,J)$.
\label{Theorem:HierarchicalBasis}
\end{Theorem}

\begin{proof}%
As in the proof of the Main Theorem, we may suppose that  $X$ does not contain any loops:
if $x$ is a loop, it is not contained in any basis, but $x$ is contained in every flat and always externally active. Hence, the removal of a loop changes neither $\Pcal(X,k,J)$ nor $\Gamma(X,k,J)$. 

\noindent
The remainder of this proof is split into four parts:
\begin{asparaenum}[\itshape (i)]
 \item Let $x\in X$ be the minimal element.
Let $B,B'\in \BB(X)$ with $x\not\in B$ and $x\in B'$.
 $x$ is externally active with respect to $B$ if and only if $x$ is a loop and $x$ is internally active in $B'$ if and only if
 $x$ is a coloop. \label{enum:HierarchicalBasisProofObservation}
 \item \emph{$\card {\Gamma(X,k,J)} = \dim \Pcal(X,k,J)$:}  we prove this by induction over the number of elements that are not coloops.
Suppose that $X$ contains only coloops. In this case there is only one basis and all its elements are internally active.
The spanning set given in \eqref{eq:PspaceGeneralPosition} is a basis and it coincides with $\Bcal(X,k,J)$.
\label{enum:CorrectNumberOfBases}

Now suppose that there is at least one element in $x$ which is not a coloop.
In addition, we may assume $\dim U\ge 1$. If $\dim U=1$, the statement follows from Remark~\ref{Prop:PspaceOneDim}.
As $\dim \Pcal(X,k,J)$ and by induction also $\card{\Gamma(X/x,k,J/x)}$ and $\card{\Gamma(X\setminus x,k,J \setminus x)}$ are independent of the order on $X$, we may assume that $x$ is the minimal element.

$\BB(X)$ can be partitioned as $\BB(X)=\BB(X\setminus x) \DisjUnion \iota(\BB(X/x))$, where $\iota$ denotes the map that sends a basis $\bar B\in \BB(X/x)$ to $B\cup x$.
It follows from {(\ref{enum:HierarchicalBasisProofObservation})} and Lemma~\ref{Lemma:CharFuncDelCon} that $\Gamma(X,k,J)$ can also be written as a disjoint union of two sets:  $\Gamma(X,k,J)=
\Gamma(X\setminus x,k, J\setminus x) \DisjUnion \iota_1(\Gamma(X / x,k, J/x))$,
 where $\iota_1$ denotes the map that sends $(\bar B,\bar I,\vek a_{\bar I})$ to $(B\cup x,I, \vek a_I)$.
Comparing this with Proposition~\ref{Prop:ExactSequence}, we see that $\card{\Gamma(X,k,J)}=%
\dim \Pcal(X,k,J)$.
\item \emph{$\Bcal(X,k,J)\subseteq S(X,k,J) \subseteq\Pcal(X,k,J)$:} if $Y=(X\setminus (B\cup E(B)))\cup I$, then $X\setminus Y=(B\cup E(B)) \setminus I$. Hence, by comparison of \eqref{eq:DefinitionGamma} and \eqref{eq:DefinitionPgenerators}, 
the statement follows.
\item 
\emph{$\Bcal(X,k,J)$  is linearly independent:} 
 by \cite[Proposition 4.21]{ardila-postnikov-2009}, $\Bcal(X,k,\Lcal(X))=\Bcal(X,k+1,\{X\})$ is linearly independent.
As $\Bcal(X,k,J)$ is contained in this set, it is also linearly independent.
\end{asparaenum}
\end{proof}

\begin{Remark}
\label{Remark:InternalBases}
We do not know if there is a simple method to construct bases for $\Pcal(X,-1,J)$.
This difficulty was already observed for the internal case by Holtz and Ron \cite{holtz-ron-2011}.
In Section~\ref{subsect:semiIntHilb}, we define a set of semi-internal bases $\BB_-(X,J)\subseteq \BB(X)$ 
whose cardinality is in some cases equal to 
the dimension of $\Pcal(X,-1,J)$. A natural candidate for $\Bcal(X,-1,J)$ would 
be the set $\tilde\Bcal(X,-1,J) := \{ p_{X\setminus (B\cup E(B))} : B\in \BB_-(X,J)\}$. 
In some cases, this is indeed a basis, but in general it has the wrong cardinality or it fails to be contained in $\Pcal(X,-1,J)$ 
(see Example~\ref{Example:NoCanonicalInternalBasis}).
\end{Remark}

\begin{Remark}
The external space has a vector space decomposition
\begin{align}
\Pcal(X,1,\{X \})= \bigoplus_{C\in \Lcal(X)}  \Pcal(X)_C
\label{eq:ExternalDecomposition}
\end{align}
where $\Pcal(X)_C:= \spa\{
 p_Y : \clos(X\setminus Y)=C \} = p_{X\setminus C} \Pcal(C,0,\{C\})$ \cite{berget-2010, orlik-terao-1994}.
This decomposition can be used to deduce 
Theorem \ref{Theorem:HierarchicalBasis} 
for $k=0$ 
from the well-known fact that $\{ p_{X \setminus (B \cup E(B)) \cup I} : B\in \BB(X),\,I\subseteq B \}$
is a basis for $\Pcal(X,1,\{X\})$ \cite{ardila-postnikov-2009, berget-2010, holtz-ron-2011}.
\label{Rem:externalPdecompBases}
\end{Remark}
\begin{Remark}
Corrado De Concini, Claudio Procesi, and Mich\`ele Vergne 
defined  the \emph{remarkable space} or 
\emph{generalized Dahmen-Micchelli space} $\Fcal(X)$
\cite{concini-procesi-book,concini-procesi-vergne-2010}.
This space also has a decomposition in terms of the lattice of flats $\Lcal(X)$.
If $X$ is unimodular, then   $\Pcal(X)_C$ and the summand of  $\Fcal(X)$ that corresponds to the flat $C$
have the same dimension. Furthermore, the two spaces are connected via the duality between $\Pcal$ and $\Dcal$-spaces.
\label{Rem:RemarkableSpaces}
\end{Remark}

\section{Hilbert series}
\label{section:HilbertSeries}
In this section, we give several formulas for the Hilbert series of $\Pcal(X,k,J)$.
The formulas in the first subsection are recursive. %
In the second subsection, we give combinatorial 
formulas for the case $k\ge 0$. The last subsection is devoted to the case
$k=-1$. All formulas only depend on the matroid $\Mfrak(X)$, the integer $k$,  and the upper set $J$, but not on the representation $X$.

\subsection{Recursive formulas}

In this subsection, we give recursive formulas for the calculation of $\hilb(\Pcal(X,k,J),t)$.
The following statement is a direct consequence of Proposition~\ref{Prop:ExactSequence} and of the Main Theorem:
\begin{Corollary}
\label{Th:HilbDelContraction}
\sameterminologyasinmaindefinition
Let $x\in X$ be an element that is not a coloop. For $k=-1$, we assume in addition that $J\supseteq \Hcal$ or $J=\{X\}$, \ie
$J$ contains either all or no hyperplanes.
Then,
 \begin{align}
  \hilb(\Pcal(X,k,J),t) = \begin{cases}
			     \hilb(\Pcal(X\setminus x,k,J\setminus x),t) & \text{if $x$ is a loop} \\
                             t \hilb(\Pcal(X\setminus x,k,J \setminus x),t) \\
				\qquad  + \hilb(\Pcal(X/x,k,J/x),t) & \text{otherwise} 
                            \end{cases} 
 \end{align}
\end{Corollary}
For coloops, the situation is more complicated and requires an additional definition.
Fix a coloop $x\in X$. Then, $X\setminus x$ is a hyperplane and the following is an upper set: 
\begin{align}
 \widehat{J/x}:= \{  \bar C : x\not \in C \in J   \} \cup \{\overline{X\setminus x}\} \subseteq \Lcal(X/x)
\end{align}
$\widehat{J/x}$ forgets 
about the flats containing $x$, whereas $J/x$ forgets about the flats not containing $x$. While the latter is
 always an upper set in $\Lcal(X/x)$, some elements of %
$\widehat{J/x}$ 
are not closed unless $X\setminus x$ is  a hyperplane.
\begin{Theorem}
\sameterminologyasinmaindefinition
Let $x\in X$ be a coloop and $k\ge 0$. Then,

\begin{align}
 \hilb(\Pcal(X,k,J),t) & = \begin{cases}
				   \sum_{j=0}^{k} t^{j+1} \hilb(\Pcal(X / x,k-j, \widehat{J / x}),t) 
				\\ \qquad + \hilb(\Pcal(X / x,k, J / x),t) & \hspace*{-3mm}\text{if } X\setminus x \in J \\
				 \sum_{j=0}^{k-1} t^{j+1} \hilb(\Pcal(X / x,k-j, \widehat{J / x}),t) \\
                                 \qquad + \hilb(\Pcal(X / x,k, J / x),t)  & \hspace*{-3mm}\text{if } X\setminus x\not \in J  
                            \end{cases} %
\label{eq:HilbertColoopDecomposition}
\end{align}
For $k=-1$, we have
\begin{align}
  \hilb(\ker \Ical(X,-1,J),t) &= \begin{cases}
                             \hilb(\ker \Ical(X/x,k,{J/x})),t) 	& \text{if } X\setminus x \in J \\
		             0			& \text{if } X\setminus x\not \in J  
                            \end{cases}
\end{align}
This formula holds for arbitrary non-empty upper sets $J\subseteq \Lcal(X)$.
\label{Thm:HilbertRecursion}
\end{Theorem}
For an example, see Example~\ref{Example:coloopRecursion}.
Actually, we prove a more  general statement, namely decomposition formulas for the $\Pcal$-spaces of type
\begin{align}
\Pcal(X,k,J) \cong \Pcal(X/x, k ,J/x) \oplus \bigoplus_j p_x^{j+1} \Pcal(X / x,k-j, \widehat{J / x})
\end{align}

\begin{proof}[Proof of Theorem~\ref{Thm:HilbertRecursion}]
We  first prove the equation for $k\ge 0$ using
 Theorem~\ref{Theorem:HierarchicalBasis} by showing that there exists a bijection between the bases of the $\Pcal$-spaces appearing on each side.
 
Fix a basis $B\in\BB(X)$. Let $\Gamma_{B}(X,k,J) := \{ (B,I,\vek a_I) \in \Gamma(X,k,J)\}$.
Since $x$ is a coloop, $x$ is contained in every basis 
and always internally active.
Hence, $\overline{I(B)} = I(B /  x) \cup \bar x$. 
A similar relationship between the sets of externally active elements with respect to $B$ and $B/x$ does not exist.
However, we do not need this since
$\clos((B\cup E(B)) \setminus I) = \clos(B \setminus I)$ for all $I\subseteq I(B)$ \cite[(7.13)]{bjoerner-1992}.

Consider the following map:
\begin{align}
\Phi_B:
 \Gamma_B(X,k,J) &\to \Gamma_B(X/x,k,J/x) \DisjUnion \BigDisjUnion_{j=0}^{ k - \eps} \Gamma_B( X/x, k-j, \widehat{J/x} ) \\
 (B,I,\vek a_I) &\mapsto \begin{cases}
                        (\overline{B\setminus x},\bar{I},\vek a_{\bar I}) \in \Gamma_B(X/x,k,\widehat{J/x})  
			    & \text{if } x\not\in I \\
                        (\overline{B\setminus x},\overline{I\setminus x},\vek a_{\bar I}) \in \Gamma_B(X/x, k -  a_x, 
                        \widehat{J/x}) & \text{if } x\in I  
                       \end{cases}
\end{align}
where $\eps=1$ if $X\setminus x\not\in J$ and $0$ otherwise and 
$\vek a_{\bar I}$ denotes the restriction of $\vek a_I$ to  $\N^{I\setminus x}$.

\noindent 
From following three facts, we can deduce that $\Phi_B$ is a bijection:
\begin{compactenum}[(i)]
 \item If $x\not \in I$, then by Lemma~\ref{Lemma:CharFuncDelCon},
$\chi_J(B \setminus I) = \chi_{J/x}(B \setminus (I\cup x))$. 
 \item If $x\in I$ and $X\setminus x\in J$ (\ie we are in the first case of \eqref{eq:HilbertColoopDecomposition}), then $\chi_J(B \setminus I) = \chi_{\widehat{J/x}}(\overline{B \setminus I})$.
 \item If $x\in I$ and $X\setminus x\not\in J$ (\ie we are in the second case of \eqref{eq:HilbertColoopDecomposition}), then $\chi_J(B \setminus I) = \chi_{\widehat{J/x}}(\overline{B \setminus I})=0$.
\end{compactenum}
We have to distinguish the cases $X\setminus x \in J$ and $X\setminus x \not\in J$
for the following reason:
if $I=\{x\}$ and $\chi_J(B \setminus I)=0$, 
then $\Gamma_B(X,k,J)$ contains no elements with $a_x=k$. 
However, $\Gamma_B(X,k-k,\widehat{J/x})\neq \emptyset$, since
by definition, $\chi_{\widehat{J/x}}(\overline{B})=1$.

Furthermore, note that the degrees of the polynomials corresponding to $(B,I,\vek a_I)$ and $(\overline{B\setminus x},
 \overline{I\setminus x},\vek a_{\bar I})$ differ by $a_x+1$ if $x\in I$. If $x\not\in I$, the polynomials corresponding to $(B,I,\vek a_I)$ and $(\overline{B\setminus x},\bar I,\vek a_{\bar I})$ have the same degree.
 This completes the proof for $k\ge 0$.
\medskip

Now we  consider the case $k=-1$. If $X\setminus x\not\in J$, then $\Ical(X,-1,J) = \ideal \{1\}$.
Suppose that $X\setminus x\in J$. 
Let $\eta$ be a defining normal for $X\setminus x$. Then $D_\eta\in \Ical(X,-1,J)$ and it is easy to check that
$\ker \Ical(X,-1,J) \cong \ker \Ical(X/x,-1,J/x)$.
\end{proof}

\subsection{Combinatorial formulas for $k\ge 0$} %
In this subsection, we prove several combinatorial formulas for $\hilb(\Pcal(X,k,J),t)$.
As in the case of the Tutte polynomial, there is a formula that depends on the internal and external activity of the bases of $X$.
For $k=0$, there is also a subset expansion formula  and a particularly simple formula for the dimension.

Theorem~\ref{Theorem:HierarchicalBasis} provides a method to compute the Hilbert series of a $\Pcal$-space combinatorially:
\begin{Corollary}
\sameterminologyasinmaindefinition
Let $k\ge 0$. Then,
 \begin{align}
     \hilb(\Pcal(X,k,J),t) &= \!\!\!\!\sum_{B\in \BB(X) }\!\! t^{N-r-\card{E(B)}} \left( 1 + \!\!\!\sum_{\emptyset\neq I\subseteq I(B)}
   \!\!\!\!\!\!\!
 \sum_{j=0}^{\substack{ \chi((B\cup E(B))\setminus I) \\ + k -1 }}  \!\!\!\! t^{\card{I} + j} 
 \binom{j + \card I  - 1}{\card{I} - 1} \right)
 \nonumber
 \intertext{where $E(B)$ and $I(B)$ denote the sets of externally resp.\ internally active elements. For $k=0$, this specializes to}
     \hilb(\Pcal(X,0,J),t) &= \sum_{B\in \BB(X) : } t^{N-r-\card{E(B)}} 
     \left( 1 + \sum_{\substack{\emptyset\neq I\subseteq I(B) \\ \chi((B\cup E(B))\setminus I)=1}}      t^{\card{I}} \right)
  \end{align}
\label{Cor:HilbertreihenCombinatorisch}
\end{Corollary}

\smallskip
Corollary \ref{Cor:HilbertreihenCombinatorisch} gives a formula in terms of the internal and external activity of the bases of $X$. 
For $k=0$, there is also a subset expansion formula similar to the one for the Tutte polynomial.
In the internal, central and external case, the Hilbert series
of the $\Pcal$-spaces are evaluations of the Tutte polynomial \cite{ardila-postnikov-2009}.
In particular, 
\begin{align}
  \hilb(\Pcal(X,0,\{X\}),t )  &=  t^{N-r}\!\!\sum_{\substack{A\subseteq X\\\rank(A)=r}}   \left(\frac 1t - 1\right)^{\card A - \rank(A) } 
   = t^{N-r} T_X\left(1,\frac 1t\right)
 \label{eq:CentralTutteEvaluation}
\\
  \hilb(\Pcal(X,1,\{X\}),t )  &= 
 t^{N-r}\sum_{A\subseteq X} t^{r-\rank(A) }  \left(\frac 1t - 1\right)^{\card A - \rank(A) }
 \label{eq:ExternalTutteEvaluation}
\\
 &= t^{N-r} T_X\left(1+t,\frac 1t\right) \nonumber
\end{align}
When looking at these  two formulas, one might wonder if it is possible
to find an ``interpolating'' formula for the semi-external case. Indeed,  the natural
guess works:
if $\chi(A)=1$, we take the corresponding summand from \eqref{eq:ExternalTutteEvaluation}
 and if $\chi(A)=0$, we take the corresponding summand from \eqref{eq:CentralTutteEvaluation}.
Note that the latter term is always $0$.
In the semi-internal case however, the analogous statement is false.

\begin{Theorem}
\sameterminologyasinmaindefinition

 \begin{align}
  \hilb(\Pcal(X,0,J),t ) & =   
 t^{N-r} \sum_{\substack{A\subseteq X\\ \chi(A)=1}} t^{r-\rank(A)}   \left(\frac 1t - 1\right)^{\card A - \rank(A) }  \label{eq:HilbTutteTypeExpression}
 \end{align}
\label{Th:HilbTutteTypeExpression}
\end{Theorem}

\begin{proof}%
As usual, we prove this statement by  deletion-contraction.
In this proof, we denote the polynomial on the right side of \eqref{eq:HilbTutteTypeExpression} by $T_{(X,J)}(t)$.

Let $x\in X$ be a loop and let $A\subseteq X\setminus x$. 
If $\chi_J(A)=0$, $A$ contributes neither to $T_{(X\setminus x,J \setminus x)}(t)$
nor to $T_{(X\setminus x,J \setminus x)}(t)$. If $\chi_J(A)=1$, 
$A$ contributes to $T_{(X\setminus x,J \setminus x)}(t)$
 the term
$t^{N-r-1} t^{r - \rank(A)}(1/t- 1)^{\card A - \rank(A)}=:f_A$. To $T_{(X,J)}(t)$, $A$ contributes the term $tf_A$ and $A\cup x$ contributes $t(1/t-1)f_A$. This implies $T_{(X\setminus x, J\setminus x)}(t)=T_{(X,J)}(t)$.

From now on, we suppose that $X$ does not contain any loops.
Suppose that $X$ contains only coloops. 
By Proposition~\ref{Prop:PspaceGeneralPosition},
\begin{align}
 \Pcal(X,0,J) =& \spa \{ p_Y : X\setminus Y \in J \}
\label{eq:PSpaceColoops}\\
\intertext{It is easy to see that  $T_{(X,J)}(t)$ is the Hilbert series of \eqref{eq:PSpaceColoops}:}
 T_{(X,J)}(t) %
 &= t^{N-r} \sum_{\substack{A\subseteq X\\ A\in J}} t^{r-\card A}  
  =  \sum_{\substack{A\subseteq X\\X\setminus A \in J}} t^{\card A}  \label{eq:HilbTutteColoops}
 \end{align}

Now suppose that $x\in X$ is neither a loop nor a coloop. By induction, we may suppose that 
\eqref{eq:HilbTutteTypeExpression} holds for $X/x$ and $X\setminus x$.
Using Corollary~\ref{Th:HilbDelContraction}
and Lemma~\ref{Lemma:CharFuncDelCon}, we obtain

\begin{align*}
 \hilb(\Pcal(X,0,J),t ) &= t \hilb(\Pcal(X\setminus x,k,J \setminus x),t) + \hilb(\Pcal(X/x,k,J/x),t) \\
 &= 
t^{N-r} \!\!\!\! \sum_{\substack{A \subseteq X\setminus x\\ \chi_{J\setminus x}(A) = 1 }} \!\!\! t^{r-\rank(A)}   \left(\frac 1t - 1\right)^{\card A - \rank(A) } 
 \\
&
\qquad + t^{N-r} \!\!\!\!\sum_{\substack{ \bar A \in X/x \\ \chi_{J/x}(\bar A)=1}}\!\!\! t^{(r-1)-\rank(\bar A)}   \left(\frac 1t - 1\right)^{\card{\bar A} - \rank(\bar A) } 
 \\
 &= t^{N-r} \sum_{\substack{x\not\in A\\ \chi_J(A)=1}} t^{r- \rank(A)}   \left(\frac 1t - 1\right)^{\card A - \rank(A) } 
 \displaybreak[2] \\
 & \qquad  + t^{N-r} \sum_{\substack{x\in A\\  \chi_J(A)=1}} t^{r- \rank(A)}   \left(\frac 1t - 1\right)^{\card A - \rank(A) }
 \displaybreak[1] \\
  &= 
 t^{N-r} \sum_{ \substack{A\subseteq X \\ \chi(A)=1 }} t^{r- \rank(A)}   \left(\frac 1t - 1\right)^{\card A - \rank(A) }
 = T_{(X,J)}(t) 
 \end{align*}
\end{proof}

If we set $t=1$ in  Theorem \ref{Th:HilbTutteTypeExpression}, we immediately obtain a result
which relates the dimension of $\Pcal(X,0,J)$  and the number of independent sets satisfying a certain property.
This was already proven with a different method by Holtz, Ron, and Xu \cite{holtz-ron-xu-2012}.
\begin{Corollary}
 \begin{align}
   \dim \Pcal(X,0,J) &=  \card{\{ Y\subseteq X : Y \text{ independent},\, \clos(Y)\in J \}} %
 \end{align}
\end{Corollary}

\begin{Remark}
It is possible to deduce Theorem \ref{Th:HilbTutteTypeExpression}
directly from \eqref{eq:CentralTutteEvaluation}.
This can be done using the decomposition \eqref{eq:ExternalDecomposition} of the external space.
\end{Remark}

\subsection{The case $k = -1$}
\label{subsect:semiIntHilb}

For $k=-1$, we do not know if there is such a nice formula as in Corollary~\ref{Cor:HilbertreihenCombinatorisch} or  Theorem~\ref{Th:HilbTutteTypeExpression}. 
The set $\BB_-(X,J):=\{ B\in \BB(X)  : \chi(B\setminus I(B)) =1 \} $ has in some cases several nice properties, but in general
the cardinality of $\BB_-(X,J)$ depends  on the order imposed on $X$ 
(cf. Remark \ref{Remark:InternalBases}). Consider for example a sequence $X$ of three vectors $a,b,c$ in general position in a two-dimensional vector space and the ideal $J=\{X,\{a\}\}$. Depending on the order, $\BB_-(X,J)$ may have cardinality $1$ or $2$.

Fix $C_0\in \Lcal(X)$ and set $J_{C_0}:=\{C\in \Lcal(X) : C\supseteq C_0\}$.
All maximal missing flats in $J_{C_0}$ are hyperplanes. They have unique defining normals (up to scaling).
Then $\ker \Ical(X,-1,J_{C_0})=\ker \Ical'(X,-1,J_{C_0}) = \bigcap_{x\in C_0} \Pcal(X\setminus x, 0, \{X\setminus x\})$. This was shown by Holtz, Ron, and Xu \cite{holtz-ron-xu-2012}.
They also show that for a specific order on $X$ (see below), $\card{\BB_-(X,J_{C_0})}=\dim \ker \Ical(X,-1,J_{C_0})$ and that
$\BB_-(X,J_{C_0})$ can be used to calculate the Hilbert series:

\begin{Theorem}[{\cite[p.\ 20]{holtz-ron-xu-2012}}]
\sameterminologyasinmaindefinition 
In addition, let $C_0\in \Lcal(X)$. Then,
 \begin{align}
     \hilb(\ker \Ical(X,-1,J_{C_0}),t) = \sum_{ B\in \BB_-(X,J_{C_0}) } t^{N-r-\card{E(B)}}  
 \end{align}
\label{Thm:SemiInternalHilbertSeries}
\end{Theorem}

The proof in \cite{holtz-ron-xu-2012} relies on the following construction:
 an independent spanning subset $I\subseteq C_0$ is fixed and the order is chosen \st the elements of $I$ are maximal.
This makes it difficult or impossible to adjust this proof to a more general setting.

\section{Zonotopal Cox Rings}
\label{section:cox}
In this section, we briefly describe the zonotopal Cox rings defined by 
Sturmfels and Xu  \cite{sturmfels-xu-2010} and we show that our Main Theorem can be used to generalize
a result on zonotopal Cox modules due to Ardila and Postnikov \cite{ardila-postnikov-2009}.

Fix $m$ vectors  $D_1,\ldots, D_m\in V$ and $\vek{u}=(u_1,\ldots,u_m)\in \N^m$.
Sturmfels and Xu \cite{sturmfels-xu-2010} introduced the Cox-Nagata ring 
$R^G \subseteq \K[s_1,\ldots, s_m,t_1,\ldots, t_m]$. 
This is the ring of polynomials that are invariant under 
 the action of a certain group $G$ which depends on the vectors $D_1,\ldots, D_m$.
 It is multigraded with a $\Z^{m+1}$-grading. 
For $r\ge 3$, $R^G$ is equal to the Cox ring of the variety $X_G$ which is gotten from $\mathbb P^{r-1}$ by blowing up the points $D_1,\ldots, D_m$.
Cox rings have received a considerable amount of attention in the recent literature in algebraic geometry. See \cite{laface-velasco-2009} for a survey.

Cox-Nagata rings are closely related to power ideals: %
let $\Ical_\vek{u} := \ideal\{ D_1^{u_1+1}, \ldots,\linebreak[2] D_m^{u_m+1} \}$
and let $\Ical_{d,\vek{u}}^{-1}$ denote the homogeneous component of grade $d$ of $\ker \Ical_\vek{u}$.
Then, $R^G_{(d,\vek{u})}$, the homogeneous component of $R^G$ of grade $(d,\vek u)$, is naturally isomorphic  to $\Ical_{d,\vek{u}}$ (\cite[Proposition 2.1]{sturmfels-xu-2010}).

Cox-Nagata rings are an object of great interest but in general, it is quite difficult to understand their structure.
However, for some choices of the vectors $D_1,\ldots, D_m$, we understand a natural subring of the Cox-Nagata ring very well.

Let $\Hcal=\{H_1,\ldots, H_m\}$ be the set of hyperplanes in $\Lcal(X)$.
$\mathfrak H \in \{0,1\}^{m \times N}$ denotes the non-containment hyperplane-vector matrix, \ie the 0-1 matrix 
whose $(i,j)$ entry 
is $1$ if and only if  $H_i$ does not contain $x_j$.

Sturmfels and Xu defined the following structures:
the \emph{zonotopal Cox ring}
\begin{align}
 \Zcal(X):= \bigoplus_{(d,\vek a)\in \N^{N+1}} R^G_{(d,\mathfrak H\vek a)}
\intertext{and for $\vek w\in \Z^n$ the \emph{zonotopal Cox module} of shift $\vek w$} 
 \Zcal(X,\vek w):= \bigoplus_{(d,\vek a)\in \N^{N+1}} R^G_{(d,\mathfrak H\vek a + \vek w)}
\end{align}
Fix a vector $\vek a\in \N^N$.
Let $X(\vek a)$ denote the sequence of $\sum_{i} a_i$ vectors in $U$ that is ob\-tained from $X$ by replacing each $x_i$ by $a_i$ copies of itself 
and let $\vek e:=(1,\ldots, 1)\in \N^m$.
 Ardila and Postnikov prove the following isomorphisms \cite[Proposition 6.3]{ardila-postnikov-2009}:
\begin{align}
 R^G_{(d, \mathfrak H \vek a)} &\cong \Pcal(X(\vek a), 1, \{X\})_d    
\\ R^G_{(d, \mathfrak H \vek a - \vek{e})} &\cong \Pcal(X(\vek a), 0, \{X\})_d    
\\ R^G_{(d, \mathfrak H \vek a - 2\vek{e})} &\cong \Pcal(X(\vek a), -1, \{X\})_d    
\end{align}
They prove those isomorphisms by showing a statement similar to the following lemma:
\begin{Lemma}
\sameterminologyasinmaindefinition
Let $\vek b\in \{0,1\}^\Hcal$
and let $J_{\vek b}:=\{ C\in \Lcal(X) :    b_H=1 \text{ for all } H\supseteq C  \}$, \ie
the maximal missing flats in $J_{\vek b}$ are exactly the hyperplanes that satisfy $b_H=0$.

Suppose that $\Ical(X(\vek a),k,J_\vek{b})= \Ical'(X(\vek a),k,J_\vek{b})$ for all $\vek a\in \N^N$. 
Then,%
\begin{align}
 R^G_{(d, \mathfrak H \vek a + (k-1)\vek{e} + \vek b)} &\cong (\ker \Ical(X,k,J_\vek{b}))_d   
\qquad \text{for all $d$}
\label{eq:CoxequalKernel}
\end{align}
\label{Prop:CoxequalKernel}
\end{Lemma}
Using the Main Theorem, we can deduce the following results about
\emph{hierarchical zonotopal Cox modules}:
\begin{Proposition}
We use the same terminology as in Lemma~\ref{Prop:CoxequalKernel}.
For the graded components of the \emph{semi-external zonotopal Cox module} $\Zcal(X, \mathfrak H \vek a -\vek{e} + \vek b)$, the following holds:
\begin{align}
 R^G_{(d, \mathfrak H\vek{a}  - \vek e + \vek{b})} \cong \Pcal(X(\vek{a}),0,J_\vek{b})_d 
\qquad \text{for all $d$}
\end{align}
\label{Prop:SemiExternalZonotopalCoxMod}
\end{Proposition}
\begin{Proposition}
We use the same terminology as in
Lemma~\ref{Prop:CoxequalKernel}.
Let $C_0\in \Lcal(X)$ be a fixed flat and $J_{C_0}:=\{ C\in \Lcal(X) : C\supseteq C_0\}$ (cf. Subsection
\ref{subsect:semiIntHilb}).
If 
$\vek b\in \{0,1\}^{\Hcal}$ satisfies $b_H=1$ if and only if $H\supseteq C_0$, then
for the graded components of the \emph{semi-internal zonotopal Cox module} $\Zcal(X, \mathfrak H \vek a - 2\vek{e} + \vek b))$, the following holds:
\begin{align}
 R^G_{(d, \mathfrak H\vek{a} - 2\vek{e} + \vek{b} )} \cong \ker \Ical(X(\vek{a}),-1,J_{C_0})_d 
\qquad \text{for all $d$}
\end{align}
\label{Prop:SemiInternalZonotopalCoxMod}
\end{Proposition}
Using Theorems~\ref{Th:HilbTutteTypeExpression} and \ref{Thm:SemiInternalHilbertSeries}, we can calculate the multigraded Hilbert series 
of the semi-external and the semi-internal zonotopal Cox modules: %
\begin{Corollary}
 In the setting of Proposition~\ref{Prop:SemiExternalZonotopalCoxMod}, the 
 dimension of $R^G_{(d, \mathfrak H\vek{a} - \vek e + \vek{b})}$
 equals the coefficient of $t^d$ in
\begin{align*} 
 \hilb(\Pcal(X(\vek{a}),0,J_{\vek b}),t ) 
 &=
 t^{\abs{\vek a} - r} \!\!\sum_{\substack{A\subseteq X\\ \chi(A)=1}} t^{r-\rank(A)} \!\!\!\!\!\!
  \sum_{\substack{1 \le s_i \le a_i \\ \vek s\in \N^A\!,\, x_i\in A}} \!\!\!
 \left(\prod_i \binom{a_i}{s_i}\right)  \left(\frac 1t - 1\right)^{\abs{\vek s} - \rank(A) }  
\end{align*}
where $\abs{\vek a}:=\sum_i a_i$.
\end{Corollary}

\begin{proof}
Apply  Theorem~\ref{Th:HilbTutteTypeExpression} to $X(\vek a)$.
Take into account that for every $S\subseteq X(\vek a)$, there is a unique pair $(A, \vek s)$ with $A\subseteq X$ and $\vek s\in \N^A$ \st $S$ is obtained from $A$ by
replacing each $x_i\in A$ by $s_i$ copies of itself. Furthermore, $\rank(S)=\rank(A)$. 
For a fixed pair $(A,\vek s)$, there are $\binom{a_i}{s_i}$ options to choose the corresponding vectors in $X(\vek a)$ for every $i$.
\end{proof}

\begin{Corollary}
 In the setting of Proposition~\ref{Prop:SemiInternalZonotopalCoxMod}, the 
 dimension of $R^G_{(d, \mathfrak H\vek{a} - 2\vek{e} + \vek{b})}$
 equals the coefficient of $t^d$ in
\begin{align} 
 \hilb(\ker \Ical(X(\vek{a}),-1,J_{C_0}),t ) 
 &=
 \sum_{B \in \BB_-(X,J_{C_0})} \sum_{\substack{0\le s_i \le a_i-1\\\vek s\in \N^B}} t^{e(B,\vek s)} %
\end{align}
where $e(B,\vek s):={\sum_{i\, :\, x_i\not\in E(B)} a_i 
  - r - \sum_{x_i\in B} s_i}$
\end{Corollary}

\begin{proof}
 Apply Theorem~\ref{Thm:SemiInternalHilbertSeries}. 
 Choose an order on $X(\vek a)$ that is compatible with the order on $X$, \ie
 if $x',y'\in X(\vek a)$ are copies of $x,y\in X$ ($x\neq y$) then $x'<y'$ if and only if $x<y$.
Fix a basis $B\subseteq X$. 
Now we examine the copies of $B$ that are contained in $X(\vek a)$.
All copies of the elements that are externally active with respect to $B$ in $X$ are externally active with respect to every copy of $B$ in $X(\vek a)$. Let $x\in B\subseteq X$. 
If the $i$th copy (the maximal one being the first) of $x$ in $X(\vek a)$ is chosen, 
$i-1$ copies of $x$ are externally active in $X(\vek a)$. Hence, 
for the copy of $B$ in $X(\vek a)$ that corresponds to $\vek s$, the exponent of $t$ equals  
$e(B,\vek s)$.
\end{proof}

\section{Examples}
\label{section:examples}
This section contains a large number of examples. In the first subsection, 
we give explicit examples for the various structures appearing in this paper
($X$, $J$, $S$, $\Pcal$, $\Ical$, $\Gamma$, $\Bcal$, $\BB$, $\BB_-$).
In the second subsection, we give an example for deletion and contraction as defined in Section~\ref{subsection:delcon}. 
In the third subsection, we exemplify the 
decomposition of $\Pcal$-spaces that appears in the proof of %
Theorem~\ref{Thm:HilbertRecursion}.
In the last subsection, we explain several problems that occur in the semi-internal case.%

In this section, we make the following identifications: $\sym(V)=\sym(U)=\K[x,y]$ respectively $\sym(V)=\sym(U)=\K[x,y,z]$.

\subsection{Structures}
\begin{align}
\text{Let } 
X_1 &:= \begin{bmatrix}
    1 & 0 & 1 \\
    0 & 1 & 1
 \end{bmatrix} = (x_1,x_2,x_3). 
\end{align}
Define two ideals $J_1:=\{X_1\}$ and $J_2:=\{X_1, (x_1), (x_3) \}$. The set of bases is
$\BB(X_1)=\{(x_1x_2), (x_1x_3), (x_2x_3) \}$. The sets of semi-internal bases are $\BB_-(X_1,J_1)=\{ (x_1x_2) \}$ and
$\BB_-(X_1,J_2)=\{ (x_1x_2), (x_1x_3) \}$.

\begin{align*}
  S(X_1,-1, J_1) &= \{ 1  \} 
& S(X_1,0, J_1) &= \{ 1,\, p_{x_1},\, p_{x_2},\, p_{x_3}  \}  \\
  \Pcal(X_1,-1,J_1) &= \spa\{1\} 
& \Pcal(X_1, 0, J_1) &= \spa\{ 1, x, y \} \\
  \Ical(X_1,-1,J_1) &= \ideal \{ x, y  \} 
& \Ical(X_1,0,J_1) &= \ideal \{ x^2, xy, y^2  \} \\
&& \Gamma(X_1,0,J_1) &= \{ ((x_1x_2),\emptyset,0),((x_1x_3),\emptyset,0), \nonumber\\
        &&& \qquad ((x_2x_3),\emptyset,0)  \} \\
&& \Bcal(X_1,0,J_1) &= \{ p_\emptyset, p_{x_2}, p_{x_1}   \}   \displaybreak[2]\\[1mm]
 S(X_1,-1, J_2) &= \{ 1, p_{x_2}  \} 
& S(X_1,0, J_2) &= \{ 1,\, p_{x_1},\, p_{x_2},\, p_{x_3},\, p_{x_1x_2},\, p_{x_2x_3}   \}  \\
 \Pcal(X_1,-1,J_2) &= \spa\{1,y\} 
& \Pcal(X, 0, J_2) &= \spa \{ 1, x, y, xy, y^2 \} \\
 \Ical(X_1,-1, J_2) &= \ideal \{ x, y^2  \} 
& \Ical(X_1,0,J_2) &= \ideal \{ x^2, xy^2, y^3  \} \\
&& \Gamma(X_1,0,J_2) &= \{ ((x_1x_2),\emptyset,0),((x_1x_3),\emptyset,0), \nonumber\\
		&&&\qquad ((x_1x_3),(x_3),0),((x_2x_3),\emptyset,0), \nonumber\\
                &&&\qquad  ((x_2x_3),(x_2), \emptyset,0)  \} \\
&& \Bcal(X_1,0,J_2) &= \{  p_\emptyset,\, p_{x_2},\, p_{x_2x_3},\, p_{x_1},\, p_{x_1x_2}    \}
\end{align*}

\subsection{Deletion and contraction}
\label{subsect:delconEx}
In this subsection, we give examples explaining  deletion and contraction for pairs $(X,J)$.
\begin{align}
 X_1\setminus x_1 &= \begin{bmatrix}
                           0 & 1 \\ 1 & 1 
                          \end{bmatrix} = (x_2,x_3) &
 X_1 /x_1 &= \begin{bmatrix}
                         1 &  1                         
                    \end{bmatrix} = (\bar x_2,\bar x_3) \\
 J_1 \setminus x_1 &= \{  (x_2,x_3 )  \} 
& J_1 / x_1 &= \{  (\bar x_2, \bar x_3 )  \} 
\\
 J_2 \setminus x_1 &= \{ (x_2, x_3), (x_3) \} 
 &  J_2 / x_1 &= \{ (\bar x_2,\bar x_3), \bar\emptyset \} =\Lcal(X_1/x_1) 
\end{align}
Recall that we identify $\K[x,y]/x$ and $\K[y]$. Then,

\begin{align*}
 \Ical(X_1\setminus x_1, 0, J_1\setminus x_1)&= \ideal\{ x, y \} 
& \Ical(X_1\setminus x_1, 0, J_2\setminus x_1)&= \ideal\{ x,y^2 \} \\
 \Pcal(X_1\setminus x_1, 0, J_1\setminus x_1)&= \spa \{1 \} 
& \Pcal(X_1\setminus x_1, 0, J_2\setminus x_1)&= \spa \{1, y \} \\
 \Ical(X_1/ x_1, 0, J_1/x_1)&= \ideal\{ y^2 \}
& \Ical(X_1/ x_1, 0, J_2/x_1)&= \ideal \{ y^3 \} \\ %
 \Pcal(X_1/ x_1, 0, J_1/x_1)&= \spa\{ 1,y \} 
& \Pcal(X_1/ x_1, 0, J_2/x_1)&= \spa\{ 1,y,y^2 \} \\
\end{align*}
The reader should check that
$\Pcal(X_1,0,J_i)= p_x \Pcal(X_1\setminus x_1, 0, J_i\setminus x_1) \oplus \Pcal(X_1/ x_1, 0, J_i / x_1)$ %
holds for $i=1$ and $i=2$.

%

\subsection{Recursion for the Hilbert series}
\begin{Example}
\label{Example:coloopRecursion}
This is an example for 
the description of $\Pcal$-spaces in  Proposition~\ref{Prop:PspaceGeneralPosition} and for 
the decomposition of $\Pcal$-spaces that appears in the proof of Theorem~\ref{Thm:HilbertRecursion}:
 \begin{align}
 X_2 &:= %
 			  \begin{bmatrix}
                           1 & 0 \\ 0 & 1 
                          \end{bmatrix} = (x_1,x_2) %
 \qquad J_3 := \{X_2\} \qquad J_4 := \{ X_2, ( x_1)  \} \\
 \Ical(X_2, 2,J_3) &= \ideal \{ x^3 , y^3, x^2y^2    \} \qquad \widehat{J_3/x_2} = \widehat{J_4/x_2}= \{(x_1)\}\\
 \Pcal(X_2, 2,J_3) &= \spa \{ 1,x,y,\; x^2,  xy, y^2,\; x^2y, xy^2  \} \\
  &= \spa \{ 1,x,x^2\} \oplus y\spa\{1, x, x^2\} \oplus y^2\spa\{ 1,x\}   \\
 \Ical(X_2, 2,J_4) &= \ideal \{ x^3 , y^4,  x^2y^2, xy^3   \} \\
 \Pcal(X_2, 2,J_4) &= \spa \{ 1,x,y,\; x^2,  xy, y^2,\; x^2y, xy^2, y^3  \} \\
  &= \spa\{1, x, x^2   \} \oplus y\spa\{1,x,x^2  \} \oplus y^2 \spa\{1,x   \} \oplus y^3 \spa\{1 \} \nonumber
 \end{align}
For a graphical description of the %
 decomposition, see Figure~\ref{fig:coloopRecursion}.
\end{Example}

\begin{figure}[tb]
 \centering
 \input{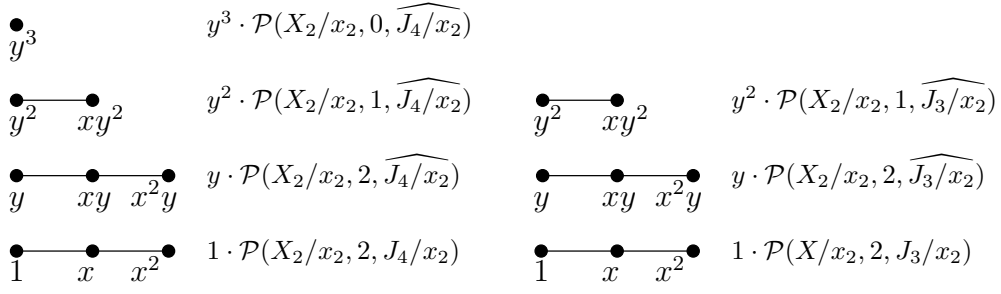}
 \caption{On the left, $\Pcal(X_2,2,J_4)$ and on the right $\Pcal(X_2,2,J_3)$. For both spaces, the decompositions corresponding to Theorem~\ref{Thm:HilbertRecursion} are shown.}
 \label{fig:coloopRecursion}
\end{figure}

\subsection{Problems in the semi-internal case}

\begin{Example}[No canonical basis for internal spaces]
\label{Example:NoCanonicalInternalBasis}
In Section~\ref{subsect:semiIntHilb}, we defined the set of semi-internal bases $\BB_-(X,J)$.
Let $\tilde\Bcal_-(X,J):=\{ p_{X\setminus (B\cup E(B))} : B\in \BB_-(X,J) \}$. This example 
shows that even in the internal case,
where $\tilde\Bcal_-(X,J)$ has the right cardinality, it is in general not contained in
 $\ker \Ical(X,-1,J)$.
 \begin{align}
   X_3 := \begin{bmatrix}
 0& 0& 1& 1& 1\\
 1& 0& 0& 0& 1\\
 0& 1& 1& 0& 0
        \end{bmatrix} &= (x_1,x_2,x_3,x_4,x_5) \\
 \BB_-(X_3,\{X_3\}) &= \{ (x_1,x_2,x_3), (x_1,x_2,x_4)   \} \displaybreak[1]\\
 \Ical(X_3,-1,\{X_3\}) &= \ideal \{ x^2,y,z   \} \displaybreak[1]\\
 \Pcal(X_3,-1,\{X_3\}) =\ker\Ical(X_3,-1,\{X_3\}) &= \spa\{ 1, x  \} \displaybreak[1]\\
 E((x_1,x_2,x_3)) = (x_4,x_5) \qquad E( (x_1,x_2,x_4) ) &= (x_5) \\
  \tilde\Bcal_-(X_3,\{X_3\}) &= \{  1, x+z \} \not \subseteq \spa\{1,x\}%
  \end{align}
\end{Example}

\begin{Example}
\label{example:InternalAdditionalCondition}
This example shows why there is an additional condition on the ideal $J$ in the Main Theorem 
for $k=-1$.
We use the matrix $X_1$ defined at the beginning of this section and the ideal $J_5:=\{X_1,\{x_1\}\}$, \ie $J \not\supseteq \Hcal$.
Then $J_5\setminus x_1 = \{X_1\setminus x_1\}$ and $J_5/x_1 = \{ X_1/x_1, \bar\emptyset \}$.  
This implies $S(X_1\setminus x_1,-1,J_5\setminus x_1)=\emptyset$, $S(X_1,-1,J_5)=\{1\}$ and $S(X_1/x_1,-1,J_5/x_1) = \{1,y\}$.
Hence, the map $\sym(\pi_{x_1}) : \Pcal(X_5,J)\to \Pcal(X_5/x_1,J_5/x_1)$ is not surjective.

The three $S$-sets appearing in this example span the corresponding kernels.
However, our proof of the Main Theorem fails here, since
$\ker \Ical(X_1,-1,J_5)\neq p_{x_1} \ker\Ical(X_1\setminus x_1, k, J_5 \setminus x_1) \oplus \ker\Ical(X_1/x_1,k,J_5/x_1)$, \ie Proposition~\ref{Prop:ExactSequence} does not hold.
%
\end{Example}

\begin{Example}
\label{Example:InternalRecursion}
This example shows why our proof of the Main Theorem
in general does not work in the case $J=\{X\}$ 
(cf. Remark~\ref{Remark:MainTheoremNotInternal}). 
It demonstrates that $\sym(\pi_{x}) : S(X,-1,\{X\}) \to S(X / x,-1, \{ X / x \})$
is in general not surjective.

Consider the following matrix:
 \begin{align}
X_4 &= 
\begin{bmatrix}
 1 &  1 & 0 & 0 & 1 & 1 & 0 \\
 1 &  1 & 1 & 1 & 0 & 0 & 0 \\
 0 &  0 & 0 & 0 & 1 & 1 & 1 \\
\end{bmatrix} = (x_1,x_2,x_3,x_4,x_5,x_6,x_7)
\intertext{The corresponding internal $\Pcal$-space and ideal are:}
\Ical(X_4,-1, \{ X_4 \}) &= \ideal \{x^3,\, y^3,\, z^2,\, (x-y)^3,\, (x-z)^2, (x-y-z)^2 \} \\
\Pcal(X_4,-1, \{ X_4 \}) &= %
 \spa \{1,\, x,\, y,\, z,\: xy+yz,\, xy+y^2,\, x^2 + xz, \nonumber \\
      &\qquad\qquad \;x^2y+xy^2+xyz+y^2z   
  \} 
\end{align}
By deletion and contraction of $x_7$ we obtain:
\begin{align}
\Ical(X_4\setminus x_7,-1,X_4\setminus x_7) &= \ideal \{x,\, y,\, z \} \\
\Pcal(X_4\setminus x_7,-1,X_4\setminus x_7) &= \spa \{ 1\} \\
 \Ical(X_4/x_7,-1,\{X_4/x_7\}) &= \ideal \{x^3,y^3, (x-y)^3\} \\
\Pcal(X_4 /x_7,-1,X_4 / x_7)&= \spa \{ 1,\,x,\,y,\; x^2,\, xy,\, y^2,\;x^2y + xy^2 \}
\end{align}
The Main Theorem and Proposition~\ref{Prop:ExactSequence} both hold in this example.

$p_{\bar x_5}p_{\bar x_6}\in S(X_4/ x_7,-1,\{X_4/x_7\})$, but $p_{x_5}p_{x_6} \not\in S(X_4,-1,\{X_4\})$!
No element of $S(X_4,-1,\{X_4\})$ is projected to $p_{\bar x_5}p_{\bar x_6}$.
Hence, our proof of the Main Theorem does not work in this case.
\end{Example}

%
%



\bibliographystyle{model1b-num-names}
\bibliography{../MasonsConjecture/Mason_Literatur}

\end{document}